\documentclass{article} 

\usepackage[T1]{fontenc}
\usepackage{emptypage}
\usepackage{centernot}
\usepackage{array}
\usepackage{pst-grad}
\usepackage{pst-arrow}
\usepackage{pst-text}

\usepackage{latexsym}
\usepackage{amsfonts}
\usepackage{amssymb}
\usepackage{amsmath}
\usepackage{amscd}
\usepackage{eucal}
\usepackage{amsthm}
\usepackage{makeidx}

\usepackage{pstricks}
\usepackage{pst-node}
\usepackage{verbatim} 
\usepackage{graphics}
\usepackage{graphicx}
\usepackage{color}   
\usepackage{framed}

\usepackage{bbm}
\usepackage{pst-plot}
\usepackage{pstricks-add}
\usepackage{hyperref}

\usepackage{xcolor}
\usepackage{mdframed}

\def\1{\ensuremath{\mathbbm{1}}}%
\def\2{\ensuremath{\mathbbm{2}}}%

\newcommand{\E}{\ensuremath{\bb{E}}}

\newcommand{\eps}{\endpspicture}
\newcommand{\ps}{\pspicture}

\newcommand{\eh}{Eckmann--Hilton}
\newcommand{\eha}{the Eckmann--Hilton argument}

\newcommand{\catequiv}{\simeq} 

\newcommand{\scr}{\scriptsize}

\renewcommand{\:}{\colon}

\newcommand{\noi}{\noindent}
\newcommand{\dend}{\ensuremath{\hfill \maltese}\end{definition}}

\newcommand{\cat}[1]{\ensuremath{\textrm{\bfseries {\sf {#1}}}}}

\newcommand{\Cat}{{\cat{Cat}}}

\newcommand{\Mon}{{\cat{Mon}}}

\newcommand{\bb}[1]{\ensuremath{\mathbb {#1}}}
\newcommand{\ed}{\end{document}}

\newcommand{\bq}{\begin{quote}}
\newcommand{\eq}{\end{quote}}
\newcommand{\bc}{\begin{center}}
\newcommand{\ec}{\end{center}}
\newcommand{\bmp}{\noi\begin{minipage}}
\newcommand{\emp}{\end{minipage}}

\newcommand{\bfr}{\begin{flushright}}
\newcommand{\efr}{\end{flushright}}

\newcommand{\lra}{\tra}

\newcommand{\vs}[1]{\vspace*{#1em}}
\newcommand{\hs}[1]{\hspace*{#1em}}

\newcommand{\polar}{\SpecialCoor} 

\newcommand{\tra}{{\psset{unit=0.1cm,nodesep=0pt} \pspicture(8,0)
\pcline{->}(1,1.1)(7,1.1) \endpspicture}}

\newcommand{\tramap}[1]{{\psset{unit=0.1cm,nodesep=0pt,labelsep=1pt} \pspicture(8,4)
\pcline{->}(1,1)(7,1)\naput[npos=0.45]{\ensuremath{\scriptstyle{#1}}} \endpspicture}}

\newcommand{\tmap}{\tramap}

\newcommand{\tlmap}[1]{{\psset{unit=0.1cm,nodesep=0pt,labelsep=1pt} \pspicture(8,4)
\pcline{<-}(1,1)(7,1)\naput[npos=0.55]{\ensuremath{\scriptstyle{#1}}} \endpspicture}}

\newcommand{\e}{\hs{0.4}=\hs{0.4}}

\newcommand{\degree}{\ensuremath{^\circ}}

\newcommand{\myps}{
\begin{small}
\pspicture
}

\newcommand{\emyps}{
\endpspicture
\end{small}
}

\newcommand\Tstrut{\rule{0pt}{2.4ex}}         

\newcommand{\myhline}{\hline\Tstrut}  

\newcommand{\vtp}[2]{\ensuremath{\setlength{\arraycolsep}{0.1ex}\hs{-0.1}\begin{array}{c}#1\\ \myhline #2 \end{array}\hs{-0.1}}}

\newcommand{\vvtp}[2]{\ensuremath{\setlength{\arraycolsep}{0.1ex}\renewcommand{\arraystretch}{1.4}\hs{-0.1}\begin{array}{c}#1\\ \myhline #2 \end{array}\hs{-0.1}}}

\newcommand{\htp}[2]{\ensuremath{\setlength{\arraycolsep}{0.6ex}\renewcommand{\arraystretch}{1.08}\hs{-0.2}\begin{array}{c|c}#1 & #2\end{array}\hs{-0.2}}}

\newcommand{\hhtp}[2]{\ensuremath{\setlength{\arraycolsep}{0.6ex}
\renewcommand{\arraystretch}{1.4}\hs{-0.2}
\begin{array}{c|c}#1 & #2\end{array}\hs{-0.2}}}

\newcommand{\threehtp}[3]{\ensuremath{\setlength{\arraycolsep}{0.6ex}\hs{-0.2}\begin{array}{c|c|c}#1 & #2 & #3\end{array}\hs{-0.2}}}

\newcommand{\threehtpscr}[3]{\ensuremath{\setlength{\arraycolsep}{0.6ex}\hs{-0.2}\begin{scriptsize}\begin{array}{c|c|c}#1 & #2 & #3\end{array}\end{scriptsize}\hs{-0.2}}}

\newcommand{\fourtp}[4]{\ensuremath{\setlength{\arraycolsep}{0.6ex}\hs{-0.2}\begin{array}{c|c}#1 & #2 \\ \myhline #3 & #4\end{array}\hs{-0.2}}}

\newcommand{\sixtp}[6]{\ensuremath{\setlength{\arraycolsep}{0.6ex}\hs{-0.2}\begin{array}{c|c|c}#1 & #2 & #3 \\ \myhline #4 & #5 & #6\end{array}\hs{-0.2}}}

\newlength\myfntht

\makeindex

\numberwithin{equation}{section}

\theoremstyle{plain}

\newtheorem{theorem}{Theorem}[section]

\newtheorem{proposition}[theorem]{Proposition}

\newtheorem{corollary}[theorem]{Corollary}

\newtheorem{lemma}[theorem]{Lemma}

\theoremstyle{definition}

\newtheorem{definition}[theorem]{Definition}

\newtheorem{example}[theorem]{Example}

\newtheorem{examples}[theorem]{Examples}
\newtheorem{nonexample}[theorem]{Non-example}
\newtheorem{remark}[theorem]{Remark}
\newtheorem{remarks}[theorem]{Remarks}
\newtheorem{exercise}[theorem]{Exercise}
\newtheorem{note}[theorem]{Note}
\newtheorem{question}[theorem]{Question}
\newtheorem{questions}[theorem]{Questions}
\newtheorem{algorithm}[theorem]{Algorithm}
\newtheorem{method}[theorem]{Method}

\renewcommand{\arraystretch}{1.05} 
\psset{
unit=1mm,
linewidth=0.7pt,
arrowlength=1.1, 
arrowsize=3pt 2,
arrowinset=0.7,
doublesep=1pt,
labelsep=2pt,
nodesep=2pt, 
dash=2pt 2pt}
  
\newcommand{\cubeunit}{\psset{unit=0.35mm}}

\newcommand{\cubebone}
{\cubeunit\pspicture(28,28)

\psframe(0,0)(20,20)
\psline(0,20)(10,27)(28,27)(28,9)(20,0)
\psline(20,20)(28,27)

\psline(10,0)(10,20)(19,27)

\psline(0,10)(20,10)(28,18)

\psline(5,23.5)(24,23.5)(24,4.5)

\rput(15,15){\scr $a$} 

\rput(-6,11){\scr $b$}

\eps}

\newcommand{\cubebtwo}
{\cubeunit\pspicture(28,28)

\psframe(0,0)(20,20)
\psline(0,20)(10,27)(28,27)(28,9)(20,0)
\psline(20,20)(28,27)

\psline(10,0)(10,20)(19,27)

\psline(5,23.5)(24,23.5)(24,4.5)

\rput[B](15,8.5){\scr $a$} 

\rput(12,25.2){\scr $b$} 

\eps}

\newcommand{\cubebthree}
{\cubeunit\pspicture(28,28)

\psframe(0,0)(20,20)
\psline(0,20)(10,27)(28,27)(28,9)(20,0)
\psline(20,20)(28,27)

\psline(10,0)(10,20)(19,27)

\psline(0,10)(20,10)(28,18)

\psline(5,23.5)(24,23.5)(24,4.5)

\rput(15,5){\scr $a$} 

\rput(12,25.2){\scr $b$} 

\eps}

\newcommand{\cubebfour}
{\cubeunit\pspicture(28,28)

\psframe(0,0)(20,20)
\psline(0,20)(10,27)(28,27)(28,9)(20,0)
\psline(20,20)(28,27)

\psline(0,10)(20,10)(28,18)

\psline(5,23.5)(24,23.5)(24,4.5)

\rput(10,5){\scr $a$} 

\rput(26,20){\scr $b$} 

\eps}

\newcommand{\cubebfive}
{\cubeunit\pspicture(28,28)

\psframe(0,0)(20,20)
\psline(0,20)(10,27)(28,27)(28,9)(20,0)
\psline(20,20)(28,27)

\psline(10,0)(10,20)(19,27)

\psline(0,10)(20,10)(28,18)

\psline(5,23.5)(24,23.5)(24,4.5)

\rput(5,5){\scr $a$} 

\rput(26,20){\scr $b$} 

\eps}

\newcommand{\cubebsix}
{\cubeunit\pspicture(28,28)

\psframe(0,0)(20,20)
\psline(0,20)(10,27)(28,27)(28,9)(20,0)
\psline(20,20)(28,27)

\psline(10,0)(10,20)(19,27)

\psline(5,23.5)(24,23.5)(24,4.5)

\rput[B](5,8.5){\scr $a$} 

\rput(26,15.5){\scr $b$} 

\eps}

\newcommand{\cubebseven}
{\cubeunit\pspicture(28,28)

\psframe(0,0)(20,20)
\psline(0,20)(10,27)(28,27)(28,9)(20,0)
\psline(20,20)(28,27)

\psline(10,0)(10,20)(19,27)

\psline(0,10)(20,10)(28,18)

\psline(5,23.5)(24,23.5)(24,4.5)

\rput(5,15){\scr $a$} 

\rput(26,11){\scr $b$} 

\eps}

\newcommand{\cubebeight}
{\cubeunit\pspicture(28,28)

\psframe(0,0)(20,20)
\psline(0,20)(10,27)(28,27)(28,9)(20,0)
\psline(20,20)(28,27)

\psline(0,10)(20,10)(28,18)

\psline(5,23.5)(24,23.5)(24,4.5)

\rput(10,15){\scr $a$} 

\rput(26,11){\scr $b$} 

\eps}

\newcommand{\cubeaone}
{\cubeunit\pspicture(28,28)

\psframe(0,0)(20,20)
\psline(0,20)(10,27)(28,27)(28,9)(20,0)
\psline(20,20)(28,27)

\psline(10,0)(10,20)(19,27)        
\psline(0,10)(20,10)(28,18)        

\rput(5,5){\scr $b$} 
\rput(15,15){\scr $a$} 

\eps}

\newcommand{\cubeathree}
{\cubeunit\pspicture(28,28)

\psframe(0,0)(20,20)
\psline(0,20)(10,27)(28,27)(28,9)(20,0)
\psline(20,20)(28,27)

\psline(10,0)(10,20)(19,27)        
\psline(0,10)(20,10)(28,18)        

\rput(15,5){\scr $a$} 
\rput(5,15){\scr $b$} 

\eps}

\newcommand{\cubeafive}
{\cubeunit\pspicture(28,28)

\psframe(0,0)(20,20)
\psline(0,20)(10,27)(28,27)(28,9)(20,0)
\psline(20,20)(28,27)

\psline(10,0)(10,20)(19,27)        
\psline(0,10)(20,10)(28,18)        

\rput(5,5){\scr $a$} 
\rput(15,15){\scr $b$} 

\eps}

\newcommand{\cubeaseven}
{\cubeunit\pspicture(28,28)

\psframe(0,0)(20,20)
\psline(0,20)(10,27)(28,27)(28,9)(20,0)
\psline(20,20)(28,27)

\psline(10,0)(10,20)(19,27)        
\psline(0,10)(20,10)(28,18)        

\rput(15,5){\scr $b$} 
\rput(5,15){\scr $a$} 

\eps}

\newcommand{\cubeatwo}
{\cubeunit\pspicture(28,28)

\psframe(0,0)(20,20)
\psline(0,20)(10,27)(28,27)(28,9)(20,0)
\psline(20,20)(28,27)

\psline(10,0)(10,20)(19,27)        

\rput[B](5,8.5){\scr $b$} 
\rput[B](15,8.5){\scr $a$} 

\eps}

\newcommand{\cubeasix}
{\cubeunit\pspicture(28,28)

\psframe(0,0)(20,20)
\psline(0,20)(10,27)(28,27)(28,9)(20,0)
\psline(20,20)(28,27)

\psline(10,0)(10,20)(19,27)        

\rput[B](5,8.5){\scr $a$} 
\rput[B](15,8.5){\scr $b$} 

\eps}

\newcommand{\cubeafour}
{\cubeunit\pspicture(28,28)

\psframe(0,0)(20,20)
\psline(0,20)(10,27)(28,27)(28,9)(20,0)
\psline(20,20)(28,27)

\psline(0,10)(20,10)(28,18)        

\rput(10,15){\scr $b$} 
\rput(10,5){\scr $a$} 

\eps}

\newcommand{\cubeaeight}
{\cubeunit\pspicture(28,28)

\psframe(0,0)(20,20)
\psline(0,20)(10,27)(28,27)(28,9)(20,0)
\psline(20,20)(28,27)

\psline(0,10)(20,10)(28,18)        

\rput(10,15){\scr $a$} 
\rput(10,5){\scr $b$} 

\eps}

\newcommand{\cubec}
{\cubeunit\pspicture(28,28)

\psframe(0,0)(20,20)
\psline(0,20)(10,27)(28,27)(28,9)(20,0)
\psline(20,20)(28,27)

\psline(5,23.5)(24,23.5)(24,4.5) 

\rput(10,10){\scr $a$} 

\rput(26,15.5){\scr $b$} 

\eps}

\newcommand{\cubemapunit}{\psset{unit=0.3mm,arrows=-}}

\newcommand{\cubemapsix}
{\cubemapunit\pspicture(28,28)

\psframe(0,0)(20,20)
\psline(0,20)(10,27)(28,27)(28,9)(20,0)
\psline(20,20)(28,27)

\psline(0,10)(20,10)(28,18)        

\rput(10,15){\scr $r$} 
\rput(10,5){\scr $l$} 

\eps}

\newcommand{\cubemapsixb}
{\cubemapunit\pspicture(28,28)

\psframe(0,0)(20,20)
\psline(0,20)(10,27)(28,27)(28,9)(20,0)
\psline(20,20)(28,27)

\psline(0,10)(20,10)(28,18)        

\rput(10,15){\scr $l$} 
\rput(10,5){\scr $r$} 

\eps}

\newcommand{\cubemapone}
{\cubemapunit\pspicture(28,28)

\psframe(0,0)(20,20)
\psline(0,20)(10,27)(28,27)(28,9)(20,0)
\psline(20,20)(28,27)

\psline(0,10)(20,10)(28,18)        

\rput(10,15){\scr $k$} 
\rput(10,5){\scr $f$} 

\eps}

\newcommand{\cubemapfour}
{\cubemapunit\pspicture(28,28)

\psframe(0,0)(20,20)
\psline(0,20)(10,27)(28,27)(28,9)(20,0)
\psline(20,20)(28,27)

\psline(10,0)(10,20)(19,27)        

\rput[B](5,8.5){\scr $k$} 
\rput[B](15,8.5){\scr $f$} 

\eps}

\newcommand{\cubemapfive}
{\cubemapunit\pspicture(28,28)

\psframe(0,0)(20,20)
\psline(0,20)(10,27)(28,27)(28,9)(20,0)
\psline(20,20)(28,27)

\psline(10,0)(10,20)(19,27)        

\rput[B](5,8.5){\scr $\beta$} 
\rput[B](15,8.5){\scr $\tau$} 

\eps}

\newcommand{\cubemapfiveb}
{\cubemapunit\pspicture(28,28)

\psframe(0,0)(20,20)
\psline(0,20)(10,27)(28,27)(28,9)(20,0)
\psline(20,20)(28,27)

\psline(10,0)(10,20)(19,27)        

\rput[B](5,8.5){\scr $\tau$} 
\rput[B](15,8.5){\scr $\beta$} 

\eps}

\newcommand{\cubemaptwo}
{\cubemapunit\pspicture(28,28)

\psframe(0,0)(20,20)
\psline(0,20)(10,27)(28,27)(28,9)(20,0)
\psline(20,20)(28,27)

\psline(5,23.5)(24,23.5)(24,4.5) 

\rput(10,10){\scr $\beta$} 

\rput(26,15.5){\scr $\tau$} 

\eps}

\newcommand{\cubemapthree}
{\cubemapunit\pspicture(28,28)

\psframe(0,0)(20,20)
\psline(0,20)(10,27)(28,27)(28,9)(20,0)
\psline(20,20)(28,27)

\psline(5,23.5)(24,23.5)(24,4.5) 

\rput(10,10){\scr $r$} 

\rput(26,15.5){\scr $l$} 

\eps}

\newcommand{\cubemapnine}
{\cubemapunit\pspicture(28,28)

\psframe(0,0)(20,20)
\psline(0,20)(10,27)(28,27)(28,9)(20,0)
\psline(20,20)(28,27)

\psline(10,0)(10,20)(19,27)        
\psline(0,10)(20,10)(28,18)        

\rput(15,5){\scr $f$} 
\rput(5,15){\scr $k$} 

\eps}

\newcommand{\cubemapseven}
{\cubemapunit\pspicture(28,28)

\psframe(0,0)(20,20)
\psline(0,20)(10,27)(28,27)(28,9)(20,0)
\psline(20,20)(28,27)

\psline(0,10)(20,10)(28,18)

\psline(5,23.5)(24,23.5)(24,4.5)

\rput(10,15){\scr $r$} 

\rput(26,11){\scr $l$} 

\eps}

\newcommand{\cubemapeight}
{\cubemapunit\pspicture(28,28)

\psframe(0,0)(20,20)
\psline(0,20)(10,27)(28,27)(28,9)(20,0)
\psline(20,20)(28,27)

\psline(10,0)(10,20)(19,27)

\psline(5,23.5)(24,23.5)(24,4.5)

\rput[B](5,8.5){\scr $\beta$} 

\rput(26,15.5){\scr $\tau$} 

\eps}

\newcommand{\cubealeft}
{\cubeunit\pspicture(28,28)

\psframe(0,0)(20,20)
\psline(0,20)(10,27)(28,27)(28,9)(20,0)
\psline(20,20)(28,27)

\psline(10,0)(10,20)(19,27)        

\rput[B](5,8.5){\scr $a$} 

\eps}

\newcommand{\cubearight}
{\cubeunit\pspicture(28,28)

\psframe(0,0)(20,20)
\psline(0,20)(10,27)(28,27)(28,9)(20,0)
\psline(20,20)(28,27)

\psline(10,0)(10,20)(19,27)        

\rput[B](15,8.5){\scr $a$} 

\eps}

\newcommand{\cubeatop}
{\cubeunit\pspicture(28,28)

\psframe(0,0)(20,20)
\psline(0,20)(10,27)(28,27)(28,9)(20,0)
\psline(20,20)(28,27)

\psline(0,10)(20,10)(28,18)        

\rput(10,15){\scr $a$} 

\eps}

\newcommand{\cubeabottom}
{\cubeunit\pspicture(28,28)

\psframe(0,0)(20,20)
\psline(0,20)(10,27)(28,27)(28,9)(20,0)
\psline(20,20)(28,27)

\psline(0,10)(20,10)(28,18)        

\rput(10,5){\scr $a$} 

\eps}

\newcommand{\cubeafront}
{\cubeunit\pspicture(28,28)

\psframe(0,0)(20,20)
\psline(0,20)(10,27)(28,27)(28,9)(20,0)
\psline(20,20)(28,27)

\psline(5,23.5)(24,23.5)(24,4.5) 

\rput(10,10){\scr $a$} 

\eps}

\newcommand{\cubeaback}
{\cubeunit\pspicture(28,28)

\psframe(0,0)(20,20)
\psline(0,20)(10,27)(28,27)(28,9)(20,0)
\psline(20,20)(28,27)

\psline(5,23.5)(24,23.5)(24,4.5) 

\rput(26,15.5){\scr $a$} 

\eps}

\newcommand{\cubeatotal}
{\cubeunit\pspicture(28,28)

\psframe(0,0)(20,20)
\psline(0,20)(10,27)(28,27)(28,9)(20,0)
\psline(20,20)(28,27)

\rput(10,10){\scr $a$} 

\eps}

\newcommand{\tmapcube}[1]{{\psset{unit=0.1cm,nodesep=0pt} \pspicture(-6,0)(6,4)
\pcline{->}(-4,5)(4,5)\naput[npos=0.45]{\ensuremath{\scriptstyle{#1}}} \endpspicture}}

\renewcommand{\a}{\makebox[0.5em]{$a$}}
\renewcommand{\b}{\makebox[0.5em]{$b$}}
\renewcommand{\c}{\makebox[0.5em]{$c$}}
\renewcommand{\d}{\makebox[0.5em]{$d$}}
\renewcommand{\e}{\makebox[0.5em]{$e$}}
\newcommand{\f}{\makebox[0.5em]{$f$}}

\begin{document}


\title{A higher-order Eckmann--Hilton argument}

\author{Eugenia Cheng \\  School of the Art Institute of Chicago \\E-mail: info@eugeniacheng.com \\[12pt]
Alexander S. Corner\\
Sheffield Hallam University\\
E-mail: alex.corner@shu.ac.uk
}


\maketitle


\begin{abstract}
We give a higher-order higher-dimensional Eckmann--Hilton argument that is entirely algebraic. First we give an explicit argument showing that if we have two monoidal structures on a category with suitable interchange, we can derive a braiding on either of the monoidal structures. Then we show that given third monoidal structure, with suitable pairwise interchange on any pair of monoidal structures, each canonical braiding is forced to be a symmetry. As a motivating example, we show that for $n \geq 3$ any $n$-degenerate semi-strict $(n~+~1)$-category has three suitably coherent monoidal structures on its single hom-category, thus the hom-category has the structure of a symmetric monoidal category.
\end{abstract}


\setcounter{tocdepth}{2}
\tableofcontents


\section*{Introduction}
\addcontentsline{toc}{section}{Introduction}

In this paper we provide an explicit higher-dimensional Eckmann--Hilton argument for three monoidal structures on a category. Our original motivation for this is in the study of $n$-degenerate $(n+1)$-categories for $n \geq 3$, as a continuation of our previous work on doubly-degenerate tricategories \cite{cc1,cc2}. In an $n$-degenerate $(n+1)$-category the only non-trivial dimensions are the $n$- and $(n+1)$-cells; if we perform a dimension shift those become the objects and morphisms of a category which should have $n$ coherent monoidal structures on it coming from the $n$ types of composition on the original $n$-cells. For $n=2$ we have two monoidal structures, and this gives rise to a braided monoidal category, via a weak Eckmann--Hilton argument.

The basic \eh\ argument \cite{eh1} says that a set equipped with two suitably coherent monoid structures must be a commutative monoid; the second monoid structure and the coherence, in the form of an interchange law, force the two monoid structures to be the same and commutative.

There are three related ways in which we are generalising this:

\begin{enumerate}

\item More dimensions: we are moving from monoids to monoidal categories.

\item Weakness: now that we are working in categories, the equalities in the basic \eh\ argument can be specified isomorphisms.

\item Iteration: now that we are working weakly and with another dimension, more ``multiplication'' structures can fit, giving rise to higher-order versions of the Eckmann--Hilton argument.

\end{enumerate}

\noi For three or more suitbly coherent monoidal structures on a category, we know that two of them will produce a braiding; the expectation is that a third monoidal structure should force the braiding to be a symmetry, via a higher-order Eckmann--Hilton argument, and in this paper we provide that argument. We call it the ``Eckmann--Hilton Sphere'', because of its geometry. This constitutes the first main part of the paper.

The other part of the paper is our motivating example or application, where we start with any 3-degenerate (semi-strict) 4-category, and show that under the dimension shift it becomes a category with three suitably coherent monoidal structures, and thus is itself a symmetrical monoidal category.

There are various subtleties involved in this work.  The first is exactly what type of 3-monoidal category we're studying, in terms of the weakness.  For our eventual application of the results to degenerate higher categories, we have a level of weakness in mind following on from our previous work. This is crucially not the maximally weak version, as that is unnecessary for what we are trying to show, and thus it is not worth dealing with the extra difficulties it would entail.  Our approach is also not the maximally strict version, as too much structure collapses in that case, and we would expect to get only strictly commutative monoidal structures. Instead we follow our previous work and use a level of semi-strictness with weak monoidal structures but strict interchange.

Another, related, subtlety is what level abstraction we are using to express these structures: in the literature \cite{bfsv1, am1, fh1, glf1}, $2$-fold monoidal categories are sometimes defined directly, by a list of data and axioms, and sometimes more abstractly by iterating an internalisation construction. While the two approaches are equivalent, we prefer to keep the concepts separate as they play different roles in this work. We are considering a series of related types of structure as follows:

\begin{enumerate}

\item 3-degenerate 4-categories

\item 3 times internalised pseudomonoids in \cat{Cat}

\item 3-fold monoidal categories

\item symmetric monoidal categories.

\end{enumerate}

\noi The iterated internalisation approach helps to make the connection with degenerate higher categories, and also motivates the list of data and axioms for 3-fold monoidal categories; the specific data and axioms are what we then need to give the explicit Eckmann--Hilton argument.

The relationship between (1) and (3) has been studied in the classic work \cite{js1}, with more explicit arguments about braidings appearing in \cite{js3}. However, that work uses a different level of strictness and a different level of abstraction. Our level of semi-strictness makes the direct proofs much more tractable than they would otherwise be, and we find their geometry to be illuminating, so we think it is worth writing them out although the end results are in some sense already known; furthermore, we are continuing our previous work studying semi-strict $n$-categories and showing ways in which they are weak enough to give the expected structures of braided and symmetric monoidal categories.

The paper is structured as follows. The first section is all background. First we give a brief exposition of the basic Eckmann--Hilton argument for two monoid structures on a set, presented in the particular way that is most pertinent for our generalisations, with the multiplications drawn geometrically (horizontal and vertical) and the steps of the argument placed around a circle.

Then we give some background on categories weakly enriched in a 2-category, with the level of strictness we will need for both iterated enrichment and iterated internalisation. The last part of the background section presents $n$-monoidal categories via an iterated internal pseudomonoid construction, which is itself presented via weak enrichment: an internal pseudomonoid in a suitable 2-category $K$ is a degenerate category weakly enriched in $K$. In each case we construct 2-category totalities with strict functors as 1-cells; this will force strict interchange when we iterate the construction. Furthermore, we use icon-like 2-cells to ensure first of all that we have only a 2-category totality and not a 3-category, and secondly to ensure that in the degenerate case we have the appropriate notion of transformation between monoids.

Weak monoidal categories are internal pseudomonoids in the 2-category $\Cat$; taking the 2-category totality as above we obtain the 2-category of weak monoidal categories, strict monoidal functors, and monoidal transformations, and we write this 2-category as $\Mon(\Cat)$ or $\cat{MonCat}$.  The next step is internal monoids in $\Mon(\Cat)$. This definition is not new: a differently semi-strict version was introduced in \cite{bfsv1} and a weaker version in \cite{am1}.  We have chosen to write a definition out fully, partly to make explicit the non-standard level of strictness we are using, and partly to be very explicit about the axioms so that we can refer to them precisely in the proof of our main theorem.  We also find it enlightening to see exactly where the axioms come from by unravelling the iterated internalisation construction. We show that the result is a category equipped with two monoidal structures and a certain amount of compatibility which we call ``strict interchange'' for short, although there are also several other quite technical but crucial axioms. We call these structures 2-monoidal categories. We then iterate one more time to produce 3-fold internal pseudomonoids in $\Cat$; we prove that this iteration gives us a category with 3 monoidal structures, any two of which form a 2-monoidal category structure, and we call these 3-monoidal categories.

In Section~\ref{braided} we revisit the derivation of a braided monoidal category from a 2-monoidal category. This is not strictly new but we have found it enlightening to write out the construction of the braiding and the proof of the axioms explicitly, using our geometric notation (horizontal and vertical) for the two monoidal structures.

In Section~\ref{higherorder} we give the higher-order Eckmann--Hilton argument. We start with a 3-monoidal category; we can take 2 of the monoidal structures and we know we have a braiding; we show that the third monoidal structure forces the braiding to the a symmetry. This is our main theorem. We write the third monoidal structure in a third dimension, which we call ``depthal'', and the higher order argument then takes the form of a geometrically satisfying sphere, divided into 8 commuting octants arising from the 3 different directions of monoidal structure, each of which has 2 possible orderings on the elements.

Note that in \cite{bmosv1} there is reference to an ``Eckmann--Hilton Sphere''; however, that work does not appear to study the commutativity of the octants as it is focused on the weak commutativity of the monoidal structures but not the coherence of the commutativity (the braid and symmetry axioms). Various authors \cite{bat2, sy1, stew1} study higher-dimensional versions of \eha\ using operadic machinery rather than directly algebraic calculations.

In Section~\ref{degenncat} we give our motivating application, showing that any $n$-degenerate $(n+1)$-category becomes an $n$-fold monoidal category upon dimension shift, and thus, by the Eckmann--Hilton Sphere, a symmetric monoidal category. We conclude in Section~\ref{futurework} with a brief overview of future directions for this work.

\subsubsection*{How to read this paper quickly}

For the higher-order Eckmann--Hilton argument:

\begin{itemize}

\item The beginning of Section~\ref{higherorder} gives our notation for 3-fold monoidal structures.

\item Theorem~\ref{ehsphere} is the statement of the higher-order Eckmann--Hilton argument. The proof is on page~\pageref{spherediag} and shows how the \eh\ Sphere is filled in, divided into octants.

\item Lemma~\ref{ehoctant} shows the argument for each octant; Proposition~\ref{twofoldaxioms} lists all the axioms for a 2-fold monoidal category, showing the ``key axioms'' needed for each octant.

\end{itemize}

\noi For the application to degenerate higher categories:

\begin{itemize}

\item Section~\ref{weakenrichment} for the definition of weak enrichment (setting our conventions for semi-strictness).

\item Section~\ref{semistrict} for our definition of semi-strict $n$-categories by iterating weak enrichment.

\item Theorem~\ref{mainthmcat} for the theorem about $n$-degenerate $(n+1)$-categories.

\end{itemize}

\subsubsection*{Terminology and notation conventions}

\begin{itemize}

\item In general we use the terminology ``strict'' for axioms holding on the nose, ``weak'' or ``pseudo'' for up to isomorphism, and ``semi-strict'' for some combination.

\item Our specific definition of semi-strict $n$-category is given in Definition~\ref{semistrict}; the essential idea is to have weak composition but strict interchange, though there are some further subtleties.

\item We sometimes refer to semi-strict $n$-categories simply as $n$-categories, as we do not consider any other type in this work; however, note that in this work 2-categories are strict.

\item We say ``category weakly enriched in a 2-category $K$'' where elsewhere in the literature these are sometimes called bicategories enriched in $K$.

\item We use sans-serif for 2-category totalities; we do not need notation for 1-category totalities in this work.

\item As the iterative constructions risk making the notation unreadable, we use some shorthand that could be ambiguous out of context, including

\begin{itemize}

\item $\cat{Cat}(K)$ for the 2-category of categories weakly enriched in a 2-category $K$, and $\Cat^n(K)$ for the iteration.

\item $\Mon(K)$ for the 2-category of pseudomonoids in $K$, and $\Mon^n(K)$ for the iteration.

\end{itemize}

\end{itemize}

%


\section{Background}\label{background}

\subsection{Background on \eha\ on sets}
\label{backgroundeha}


The basic idea of \eha\ is that if we have a set with two monoid structures distributing over each other in an appropriate sense then they must be the same and commutative.  The basic framework can be stated directly or via internalisation:

\begin{itemize}

\item Direct: We start with a set equipped with two monoid structures sharing the same unit and satisfying interchange.

\item Internalisation: We start with a monoid internal to the category of monoids.

\end{itemize}

\noi The two frameworks are equivalent, and this is a model for what happens at the next dimension up.

\begin{theorem}[Eckmann--Hilton argument]
Let $A$ be a set equipped with two monoid structures $\circ$ and $\ast$ sharing the same unit, such that the following holds for all $a,b,c,d \in A$:
\[(a \ast b) \circ (c \ast d) = (a \circ c) \ast (b \circ d).\]
Equivalently, let $A$ be an internal monoid in the cartesian monoidal category of monoids.

Then for all $a, b, \in A$:
\begin{enumerate}
\item $a \circ b = a \ast b$, and
\item $a \ast b = b \ast a$.
\end{enumerate}
\end{theorem}

Note that weaker hypotheses are possible, for example associativity does not need to be assumed as it follows. In \cite{koc3} it is shown that associativity alone (without assuming units) is enough. However, for this work we are not concerned with stating the strongest possible result.

\begin{proof}
\[\begin{array}{rcl}
a \ast b &=& (a \circ 1) \ast (1 \circ b) \\
 &=& (a \ast 1) \circ (1 \ast b) \\
  &=& a \circ b \\
   &=& (1 \ast a) \circ (b \ast 1) \\
    &=& (1 \circ b) \ast (a \circ 1) \\
     &=& b \ast a
     \end{array}\]
     \end{proof}

We find this proof to be more illuminating when written geometrically, and this is especially helpful for the generalisation into higher dimensions. We write the two binary operations vertically and horizontally, as follows:

\[\vtp{a}{b} \hs{1} \mbox{and} \hs{1} \htp{a}{b}\]

\noi The interchange condition may then be written thus:

\[\vtp{\htp{\a}{\b}\vs{0.2}}{\htp{\c}{\d}} \hs{0.7} = \hs{0.7} \htp{\vtp{\a}{\c}\hs{0.1}}{\hs{0.1}\vtp{\b}{\d}}\]

\noi and the Eckmann--Hilton argument becomes this:

\[\htp{\a}{\b} \hs{0.7} = \hs{0.7}
\htp{\vtp{\a}{1}\hs{0.1}}{\hs{0.1}\vtp{1}{\b}}  \hs{0.7} = \hs{0.7}
\vtp{\htp{\a}{1}\vs{0.2}\hs{0.0}}{\htp{1}{\b}} \hs{0.7} = \hs{0.7}
\vtp{\a}{\b}  \hs{0.7} = \hs{0.7}
\vtp{\htp{1}{\a}\vs{0.2}}{\htp{\b}{1}} \hs{0.7} =  \hs{0.7}
\htp{\vtp{1}{\b}\hs{0.1}}{\hs{0.1}\vtp{\a}{1}}  \hs{0.7} = \hs{0.7}
\htp{\b}{\a} \]

\noi In fact there are two ``orientations'' in which one can make this argument, represented by two ways around the following ``Eckmann--Hilton clock'' from left to right: clockwise or anti-clockwise (and to streamline the notation we write the unit as a blank instead of 1).

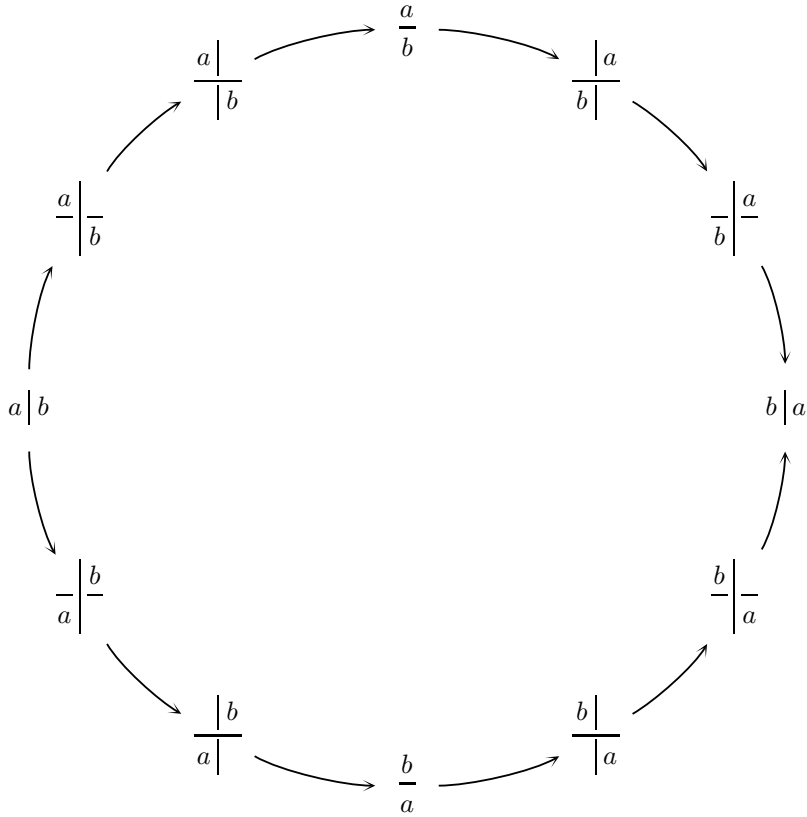
\[\pspicture(-60,-54)(60,54)
\polar

\rput(50;180){\rnode{a9}{\htp{\a}{\b}}}

\rput(50;150){\rnode{a10}{\htp{\vtp{\a}{}\hs{0.1}}{\hs{0.1}\vtp{}{\b}}}}

\rput(50;120){\rnode{a11}{\vtp{\htp{\a}{\hs{0.5}}\vs{0.2}}{\htp{\hs{0.5}}{\b}}}}

\rput(50;90){\rnode{a12}{\vtp{\a}{\b}}}

\rput(50;60){\rnode{a1}{\vtp{\htp{\hs{0.5}}{\a}\vs{0.2}}{\htp{\b}{\hs{0.5}}}}}

\rput(50;30){\rnode{a2}{\htp{\vtp{}{\b}\hs{0.1}}{\hs{0.1}\vtp{\a}{}}}}

\rput(50;0){\rnode{a3}{\htp{\b}{\a}}}

\rput(50;-30){\rnode{a4}{\htp{\vtp{\b}{}\hs{0.1}}{\hs{0.1}\vtp{}{\a}}}}

\rput(50;-60){\rnode{a5}{\vtp{\htp{\b}{\hs{0.5}}\vs{0.2}}{\htp{\hs{0.5}}{\a}}}}

\rput(50;-90){\rnode{a6}{\vtp{\b}{\a}}}

\rput(50;-120){\rnode{a7}{\vtp{\htp{\hs{0.5}}{\b}\vs{0.2}}{\htp{\a}{\hs{0.5}}}}}

\rput(50;-150){\rnode{a8}{\htp{\vtp{}{\a}\hs{0.1}}{\hs{0.1}\vtp{\b}{}}}}

\psset{arrows=->, nodesep=6pt,arcangle=15,ncurv=0.5}

\ncarc[nodesepA=8pt,nodesepB=4pt]{a9}{a10}
\ncarc[nodesepA=4pt,nodesepB=6pt]{a10}{a11}
\ncarc[nodesepA=6pt,nodesepB=10pt]{a11}{a12}
\ncarc[nodesepA=10pt,nodesepB=6pt]{a12}{a1}
\ncarc[nodesepA=6pt,nodesepB=4pt]{a1}{a2}
\ncarc[nodesepA=4pt,nodesepB=10pt]{a2}{a3}


\psset{arcangle=-15, ncurv=0.5}
\ncarc[nodesepA=10pt,nodesepB=2pt]{a9}{a8}
\ncarc[nodesepA=4pt,nodesepB=6pt]{a8}{a7}
\ncarc[nodesepA=6pt,nodesepB=10pt]{a7}{a6}
\ncarc[nodesepA=10pt,nodesepB=6pt]{a6}{a5}
\ncarc[nodesepA=6pt,nodesepB=4pt]{a5}{a4}
\ncarc[nodesepA=4pt,nodesepB=10pt]{a4}{a3}

\eps\]

At the level of sets and monoid structures every arrow is an equality, so the choice of orientation makes no difference, but in the weak higher-dimensional generalisation this becomes crucial.

Before we move on to the generalisation into higher dimensions, we need to generalise the underlying framework. While the basic Eckmann--Hilton argument involves a set with two monoid structures, we will now need a category with two monoidal category structures.


\subsection{Background on weak enrichment}
\label{weakenrichment}

In this section we will gather the definitions we need for categories weakly enriched in a cartesian 2-category $K$. We will use this in two ways: 

\begin{itemize}
\item The degenerate version will give us the definition of pseudomonoid in $K$, and we will iterate this construction to define $n$-monoidal categories.

\item We will iterate the weak enrichment construction to define semi-strict $n$-categories.

\end{itemize}

\noi Throughout this section, let $K$ be a cartesian 2-category, that is, a 2-category with products (with the strictest possible universal property). We are going to construct the 2-category of weak $K$-categories, strict $K$-functors, and $K$-icons. This slightly unusual combination of weakness and strictness is exactly what we need for our notion of semi-strict $n$-categories. Note that this is slightly weaker than the definitions we used in \cite{cc2} as in that work we used weak enrichment in the first step but strict enrichment in the second. This is a stricter special case of the definition given in \cite{cg4} of enrichment in a braided monoidal bicategory; our case is much stricter and thus much simpler, as weakening only the enrichment and not the functors does not complicate the technicalities of the 2-categorical structure compared with the fully strict case. Categories weakly enriched in a 2-category are sometimes called enriched bicategories.

\begin{definition}\label{wencat} Let $K$ be a cartesian 2-category.  A \emph{category weakly enriched in $K$} or \emph{weak $K$-category} has


\begin{itemize}
 \item an underlying $K$-graph $X$, that is

 \begin{itemize}
\item a set $X_0$ of objects, and
\item for all $a,b \in X_0$ a hom-object $X(a,b)$ which is a 0-cell of $K$,
\end{itemize}

\item composition: for all $a,b,c \in X_0$ a 1-cell in $K$
\[m_{abc} : X(b,c) \times X(a,b) \lra X(a,c),\]
\item identities: for all $a \in X_0$ a 1-cell in $K$
\[I_a: 1 \tra X(a,a),\]

\item unit constraints: for all $a, b \in X_0$ invertible 2-cells in $K$

\[\psset{unit=0.08cm,labelsep=0pt,nodesep=3pt}
\pspicture(0,-12)(100,12)

\rput(18,0){\pspicture(-10,-10)(10,10)


\rput(-32,10){\rnode{a1}{$X(a,b) \times 1$}} 
\rput(10,10){\rnode{a2}{$X(a,b) \times X(a,a)$}} 

\rput(10,-10){\rnode{b2}{$X(a,b)$}}  

\psset{nodesep=3pt,labelsep=2pt,arrows=->}
\ncline{a1}{a2}\naput[npos=0.45]{{\scriptsize $1 \hs{-0.2} \times \hs{-0.2} I_a$}} 
\ncline{a1}{b2}\nbput{{\scriptsize $\sim$}} 
\ncline{a2}{b2}\naput{{\scriptsize $m_{a,a,b}$}} 

\rput(1,1){
\psset{labelsep=1.5pt}
\pnode(2,2){a3}
\pnode(-2,-2){b3}
\ncline[doubleline=true,arrowinset=0.6,arrowlength=0.8,arrowsize=0.5pt 2.1,nodesep=0pt]{a3}{b3} \nbput[npos=0.45]{{\scriptsize $r_{a,b}$}}
}

\endpspicture}


\rput(58,0){\pspicture(-10,-10)(10,10)


\rput(52,10){\rnode{a1}{$1 \times X(a,b)$}} 
\rput(10,10){\rnode{a2}{$X(b,b) \times X(a,b)$}} 

\rput(10,-10){\rnode{b2}{$X(a,b)$}}  

\psset{nodesep=3pt,labelsep=2pt,arrows=->}
\ncline{a1}{a2}\nbput[npos=0.4]{{\scriptsize $I_b \hs{-0.2} \times \hs{-0.2} 1$}} 
\ncline{a1}{b2}\naput{{\scriptsize $\sim$}} 
\ncline{a2}{b2}\nbput{{\scriptsize $m_{a,b,b}$}} 

\rput(20,1){
\psset{labelsep=1.5pt}
\pnode(-2,2){a3}
\pnode(2,-2){b3}
\ncline[doubleline=true,arrowinset=0.6,arrowlength=0.8,arrowsize=0.5pt 2.1,nodesep=0pt]{a3}{b3} \naput[npos=0.4]{{\scriptsize $l_{a,b}$}}
}

\endpspicture}

\eps\]

\item associativity constraints: for all $a,b,c,d \in X_0$, an invertible 2-cell in $K$

\[\psset{unit=0.08cm,labelsep=0pt,nodesep=3pt}
\pspicture(-10,-12)(10,12)


\rput(-35,10){\rnode{a1}{$X(c,d) \times X(b,c) \times X(a,b)$}} 
\rput(35,10){\rnode{a2}{$X(b,d) \times X(a,b)$}} 

\rput(-35,-10){\rnode{b1}{$X(c,d) \times X(a,c)$}}   
\rput(35,-10){\rnode{b2}{$X(a,d)$}}  

\psset{nodesep=3pt,labelsep=2pt,arrows=->}
\ncline{a1}{a2}\naput[npos=0.45]{{\scriptsize $m_{b,c,d} \hs{-0.2} \times \hs{-0.2} 1$}} 
\ncline{b1}{b2}\nbput{{\scr $m_{a,c,d}$}} 
\ncline{a1}{b1}\nbput{{\scriptsize $1 \hs{-0.2} \times \hs{-0.2} m_{a,b,c}$}} 
\ncline{a2}{b2}\naput{{\scriptsize $m_{a,b,d}$}} 

\rput(1,1){
\psset{labelsep=1.5pt}
\pnode(2,2){a3}
\pnode(-2,-2){b3}
\ncline[doubleline=true,arrowinset=0.6,arrowlength=0.8,arrowsize=0.5pt 2.1,nodesep=0pt]{a3}{b3} \naput[npos=0.45]{{\scriptsize $\alpha_{a,b,c,d}$}}
}

\endpspicture\]

\end{itemize}

\noi This data satisfies the following axioms. To streamline the notation we write $(a,b)$ for $X(a,b)$, omit all subscripts and $\times$ symbols, and suppress coherence for $\times$.

\begin{itemize}
\item Associativity axiom:
\end{itemize}

\[\psset{unit=0.14cm,labelsep=2pt,nodesep=3pt}
\pspicture(-25,40)(25,115)

\rput(0,100){
\pspicture(50,50)


\rput(20,12){\rnode{a2}{$(d,e)(a,d)$}}
\rput(46,12){\rnode{a4}{$(a,e)$}}
\rput(3,23){\rnode{b1}{$(d,e)(c,d)(a,c)$}}
\rput(20,30){\rnode{c2}{$(d,e)(b,d)(a,b)$}}
\rput(46,30){\rnode{c4}{$(b,e)(a,b)$}}
\rput(3,40){\rnode{d1}{$(d,e)(c,d)(b,c)(a,b)$\hs{0.8}}}
\rput(28,40){\rnode{d3}{\hs{0.8}$(c,e)(b,c)(a,b)$}}

\ncline[nodesepA=-4pt,nodesepB=-4pt]{->}{d1}{d3} \naput{{\scriptsize $m11$}} 
\ncline{->}{d3}{c4} \naput{{\scriptsize $m1$}}
\ncline{->}{d1}{c2} \naput{{\scriptsize $1m1$}}
\ncline{->}{c2}{c4} \naput{{\scriptsize $m1$}}

\ncline{->}{d1}{b1} \nbput{{\scriptsize $11m$}}
\ncline{->}{b1}{a2} \nbput{{\scriptsize $1m$}}
\ncline{->}{c2}{a2} \naput{{\scriptsize $1m$}}

\ncline{->}{c4}{a4} \naput{{\scriptsize $m$}} 
\ncline{->}{a2}{a4} \nbput[labelsep=3pt]{{\scriptsize $m$}}


{
\rput[c](25,34){\psset{unit=1mm,doubleline=true,arrowinset=0.6,arrowlength=0.5,arrowsize=0.5pt 2.1,nodesep=0pt,labelsep=1pt}
\pcline{->}(3,3)(0,0) \naput{{\scriptsize $\alpha1$}}}}

{
\rput[c](11,25){\psset{unit=1mm,doubleline=true,arrowinset=0.6,arrowlength=0.5,arrowsize=0.5pt 2.1,nodesep=0pt,labelsep=2pt}
\pcline{->}(3,3)(0,0) \naput{{\scriptsize $1\alpha$}}}}

{
\rput[c](33,20){\psset{unit=1mm,doubleline=true,arrowinset=0.6,arrowlength=0.5,arrowsize=0.5pt 2.1,nodesep=0pt,labelsep=2pt}
\pcline{->}(3,3)(0,0) \naput{{\scriptsize  $\alpha$}}}}

\endpspicture}

\rput(-25,83){$=$}


\rput(0,60){
\pspicture(50,50)


\rput(20,12){\rnode{a2}{$(d,e)(a,d)$}}
\rput(46,12){\rnode{a4}{$(a,e)$}}
\rput(3,23){\rnode{b1}{$(d,e)(c,d)(a,c)$}}

\rput(28,23){\rnode{b3}{$(c,e)(a,c)$}}

\rput(46,30){\rnode{c4}{$(b,e)(a,b)$}}

\rput(3,40){\rnode{d1}{$(d,e)(c,d)(b,c)(a,b)$\hs{0.8}}}
\rput(28,40){\rnode{d3}{\hs{0.8}$(c,e)(b,c)(a,b)$}}

\ncline[nodesep=-4pt]{->}{d1}{d3} \naput{{\scriptsize $m11$}} 
\ncline{->}{d3}{b3} \naput{{\scriptsize $1m$}}
\ncline{->}{d1}{b1} \nbput{{\scriptsize $11m$}}
\ncline{->}{b1}{b3} \naput{{\scriptsize $m1$}}

\ncline{->}{d3}{c4} \naput{{\scriptsize $m1$}} 
\ncline{->}{c4}{a4} \naput{{\scriptsize $m$}}
\ncline{->}{b3}{a4} \naput{{\scriptsize $m$}}

\ncline{->}{b1}{a2} \nbput{{\scriptsize $1m$}} 
\ncline{->}{a2}{a4} \nbput[labelsep=3pt]{{\scriptsize $m$}}


{
\rput[c](36,26){\psset{unit=1mm,doubleline=true,arrowinset=0.6,arrowlength=0.5,arrowsize=0.5pt 2.1,nodesep=0pt,labelsep=1pt}
\pcline{->}(3,3)(0,0) \naput{{\scriptsize $\alpha$}}}}

{
\rput[c](24,16){\psset{unit=1mm,doubleline=true,arrowinset=0.6,arrowlength=0.5,arrowsize=0.5pt 2.1,nodesep=0pt,labelsep=2pt}
\pcline{->}(3,3)(0,0) \naput{{\scriptsize $\alpha$}}}}

\rput[c](16,32){\scr $=$}

\eps}

\endpspicture\]


\begin{itemize}
\item Unit axiom:

\end{itemize}

\[\psset{unit=0.14cm,labelsep=2pt,nodesep=3pt}
\pspicture(-20,15)(20,90)

\rput(0,70){
\pspicture(0,10)(50,45)


\rput(7,40){\rnode{a}{$(b,c)1(a,b)$}}
\rput(30,40){\rnode{b}{$(b,c)(b,b)(a,b)$}}
\rput(47,28){\rnode{c}{$(b,c)(a,b)$}}
\rput(30,23){\rnode{d}{$(b,c)(a,b)$}}
\rput(47,11){\rnode{e}{$(a,c)$}}

\ncline{->}{a}{b} \naput{{\scriptsize $1I1$}} 
\ncline{->}{b}{c} \naput{{\scriptsize $m1$}}
\ncline{->}{c}{e} \naput{{\scriptsize $m$}}
\ncline{->}{a}{d} \nbput{{\scriptsize $\sim$}}
\ncline{->}{d}{e} \nbput{{\scriptsize $m$}}

\ncline{->}{b}{d} \naput{{\scriptsize $1m$}}

{
\rput[c](21,34){\psset{unit=1mm,doubleline=true,arrowinset=0.6,arrowlength=0.5,arrowsize=0.5pt 2.1,nodesep=0pt,labelsep=2pt}
\pcline{->}(3,3)(0,0) \naput{{\scriptsize $1l$}}}}

{
\rput[c](37,25){\psset{unit=1mm,doubleline=true,arrowinset=0.6,arrowlength=0.5,arrowsize=0.5pt 2.1,nodesep=0pt,labelsep=2pt}
\pcline{->}(3,3)(0,0) \naput{{\scriptsize $\alpha$}}}}

\endpspicture}

\rput(-25,60){$=$}

\rput(0,30){\pspicture(50,45)


\rput(7,40){\rnode{a}{$(b,c)1(a,b)$}}
\rput(30,40){\rnode{b}{$(b,c)(b,b)(a,b)$}}
\rput(47,28){\rnode{c}{$(b,c)(a,b)$}}
\rput(47,11){\rnode{e}{$(a,c)$}}

\ncline{->}{a}{b} \naput{{\scriptsize $1I1$}} 
\ncline{->}{b}{c} \naput{{\scriptsize $m1$}}
\ncline{->}{c}{e} \naput{{\scriptsize $m$}}
\nccurve[angleA=-30,angleB=180]{->}{a}{c} \nbput{{\scriptsize $\sim$}}

{
\rput[c](26,34){\psset{unit=1mm,doubleline=true,arrowinset=0.6,arrowlength=0.5,arrowsize=0.5pt 2.1,nodesep=0pt,labelsep=2pt}
\pcline{->}(3,3)(0,0) \naput{{\scriptsize  $r1$}}}}

\eps}

\endpspicture
\]

\end{definition}

Note that one face of the associativity cube is an equality so this becomes the usual associativity pentagon in a bicategory, and the unit axiom becomes the unit triangle.

The definitions of $K$-functors and $K$-icons are not affected by the weakness of the enrichment; however, the following is a strict version of the fully weak definition given in \cite{cg4}. Note that although the 2-cell constraints are all identities, there are still some axioms at the 2-cell level.

\begin{definition}
Given weak $K$-categories $X$ and $X'$, a (strict) \emph{$K$-functor} $F: X \tra X'$ is given by 

\begin{itemize}
 \item a function $F_0\:  X_0 \tra X'_0$,

\item for all $a, b, \in X_0$ a 1-cell $F_{ab} \: X(a,b) \tra X'(Fa, Fb) \in K$
\end{itemize}
satisfying the following axioms 

\begin{itemize}
\item strict 2-cell constraints: the following diagrams commute
\end{itemize}

\[\psset{unit=0.1cm,labelsep=2pt,nodesep=3pt}
\pspicture(0,-7)(80,27)

\rput(20,5){
\pspicture(0,-5)(40,22)


\rput(0,20){\rnode{a1}{$(b,c)(a,b)$}} 
\rput(35,20){\rnode{a2}{$(Fb, Fc)(Fa,Fb)$}} 

\rput(0,0){\rnode{b1}{$(a,c)$}}   
\rput(35,0){\rnode{b2}{$(Fa,Fc)$}}  

\psset{nodesep=3pt,labelsep=3pt,arrows=->}
\ncline{a1}{a2}\naput{{\scriptsize $F_{bc}F_{ab}$}} 
\ncline{b1}{b2}\nbput{{\scriptsize $F_{ac}$}} 
\ncline{a1}{b1}\nbput{{\scriptsize $m$}} 
\ncline{a2}{b2}\naput{{\scriptsize $m'$}} 



\endpspicture}


\rput(70,5){
\psset{unit=0.1cm,labelsep=2pt,nodesep=3pt,npos=0.4}
\pspicture(0,-5)(20,22)


\rput(0,20){\rnode{a1}{$1$}}  
\rput(26,20){\rnode{a2}{$(a,a)$}}  
\rput(26,0){\rnode{a3}{$(Fa,Fa).$}}  

\ncline{->}{a1}{a2} \naput[npos=0.6]{{\scriptsize $I_a$}} 
\ncline{->}{a1}{a3} \nbput[npos=0.55]{{\scriptsize $I'_a$}} 
\ncline{->}{a2}{a3} \naput{{\scriptsize $F_{aa}$}} 


\eps}

\endpspicture
\]

\begin{itemize}
 \item associativity (we now omit the subscripts on $F$ and $I$)

\[\psset{unit=0.14cm,labelsep=2pt,nodesep=3pt}
\pspicture(50,70)

\rput(25,55){
\pspicture(0,10)(50,45)


\rput(20,12){\rnode{a2}{$(a,d)$}}
\rput(46,12){\rnode{a4}{$(Fa,Fd)$}}
\rput(0,23){\rnode{b1}{$(c,d)(a,c)$}}
\rput(20,30){\rnode{c2}{$(b,d)(a,b)$}}
\rput(46,30){\rnode{c4}{$(Fb,Fd)(Fa,Fb)$}}
\rput(0,40){\rnode{d1}{$(c,d)(b,c)(a,b)\hs{0.5}$}}
\rput(28,40){\rnode{d3}{\hs{1.5}$(Fc,Fd)(Fb,Fc)(Fa,Fb)$}}

\ncline[nodesepA=-2pt,nodesepB=-12pt]{->}{d1}{d3} \naput{{\scriptsize $FFF$}} 
\ncline{->}{d3}{c4} \naput{{\scriptsize $m'1$}}
\ncline{->}{d1}{c2} \naput{{\scriptsize $m1$}}
\ncline{->}{c2}{c4} \naput{{\scriptsize $FF$}}

\ncline{->}{d1}{b1} \nbput{{\scriptsize $1m$}}
\ncline{->}{b1}{a2} \nbput{{\scriptsize $m$}}
\ncline{->}{c2}{a2} \naput{{\scriptsize $m$}}

\ncline{->}{c4}{a4} \naput{{\scriptsize $m'$}} 
\ncline{->}{a2}{a4} \nbput[labelsep=3pt]{{\scriptsize $F$}}


{
\rput[c](25,34){\psset{unit=1mm,doubleline=true,arrowinset=0.6,arrowlength=0.5,arrowsize=0.5pt 2.1,nodesep=0pt,labelsep=1pt}
\pcline{-}(2,2)(0,0)}} 

{
\rput[c](11,25){\psset{unit=1mm,doubleline=true,arrowinset=0.6,arrowlength=0.5,arrowsize=0.5pt 2.1,nodesep=0pt,labelsep=2pt}
\pcline{->}(3,3)(0,0) \naput{{\scriptsize $\alpha$}}}}

{
\rput[c](33,20){\psset{unit=1mm,doubleline=true,arrowinset=0.6,arrowlength=0.5,arrowsize=0.5pt 2.1,nodesep=0pt,labelsep=2pt}
\pcline{-}(2,2)(0,0) \naput{{\scriptsize $$}}}}

\endpspicture}

\rput(-10,40){$=$}

\rput(25,14){\pspicture(50,44)


\rput(20,12){\rnode{a2}{$(a,d)$}}
\rput(46,12){\rnode{a4}{$(Fa,Fd)$}}
\rput(0,23){\rnode{b1}{$(c,d)(a,c)$}}

\rput(28,23){\rnode{b3}{$(Fc,Fd)(Fa,Fc)$}}

\rput(46,30){\rnode{c4}{$(Fb,Fd)(Fa,Fb)$}}
\rput(0,40){\rnode{d1}{$(c,d)(b,c)(a,b)\hs{0.5}$}}
\rput(28,40){\rnode{d3}{\hs{1.5}$(Fc,Fd)(Fb,Fc)(Fa,Fb)$}}

\ncline[nodesepA=-2pt,nodesepB=-12pt]{->}{d1}{d3} \naput{{\scriptsize $FFF$}} 

\ncline{->}{d3}{b3} \naput{{\scriptsize $1m'$}}
\ncline{->}{d1}{b1} \nbput{{\scriptsize $1m$}}
\ncline{->}{b1}{b3} \naput{{\scriptsize $FF$}}

\ncline{->}{d3}{c4} \naput{{\scriptsize $m'1$}} 
\ncline{->}{c4}{a4} \naput{{\scriptsize $m'$}}
\ncline{->}{b3}{a4} \naput{{\scriptsize $m'$}}

\ncline{->}{b1}{a2} \nbput{{\scriptsize $m$}} 
\ncline{->}{a2}{a4} \nbput[labelsep=3pt]{{\scriptsize $F$}}


{
\rput[c](36,26){\psset{unit=1mm,doubleline=true,arrowinset=0.6,arrowlength=0.5,arrowsize=0.5pt 2.1,nodesep=0pt,labelsep=1pt}
\pcline{->}(3,3)(0,0) \naput{{\scriptsize $\alpha'$}}}}

{
\rput[c](24,16){\psset{unit=1mm,doubleline=true,arrowinset=0.6,arrowlength=0.5,arrowsize=0.5pt 2.1,nodesep=0pt,labelsep=2pt}
\pcline{-}(2,2)(0,0) \naput{{\scriptsize $$}}}}

{
\rput[c](13,31){\psset{unit=1mm,doubleline=true,arrowinset=0.6,arrowlength=0.5,arrowsize=0.5pt 2.1,nodesep=0pt,labelsep=2pt}
\pcline{-}(2,2)(0,0) \naput{{\scriptsize $$}}}}

\eps}

\endpspicture\]

\item right unit

\[\psset{unit=0.12cm,labelsep=2pt,nodesep=3pt}
\pspicture(50,70)

\rput(25,55){\pspicture(0,8)(50,40)

%

\rput(20,12){\rnode{a2}{$(a,b)$}}
\rput(46,12){\rnode{a4}{$(Fa,Fb)$}}

\rput(20,30){\rnode{c2}{$(a,b)(a,a)$}}
\rput(46,30){\rnode{c4}{$(Fa,Fb)(Fa,Fa)$}}
\rput(0,40){\rnode{d1}{$(a,b)1$}}
\rput(28,40){\rnode{d3}{$(Fa,Fb)1$}}

\ncline{->}{d1}{d3} \naput{{\scriptsize $F1$}} 
\ncline{->}{d3}{c4} \naput{{\scriptsize $1I'$}}
\ncline{->}{d1}{c2} \naput{{\scriptsize $1I$}}
\ncline{->}{c2}{c4} \naput{{\scriptsize $FF$}}

\ncline{->}{d1}{a2} \nbput{{\scriptsize $\sim$}}
\ncline{->}{c2}{a2} \naput{{\scriptsize $m$}}

\ncline{->}{c4}{a4} \naput{{\scriptsize $m'$}} 
\ncline{->}{a2}{a4} \nbput[labelsep=3pt]{{\scriptsize $F$}}


{
\rput[c](25,34){\psset{unit=1mm,doubleline=true,arrowinset=0.6,arrowlength=0.5,arrowsize=0.5pt 2.1,nodesep=0pt,labelsep=1pt}
\pcline{-}(2,2)(0,0) \naput{{\scriptsize $$}}}}

{
\rput[c](14,24){\psset{unit=1mm,doubleline=true,arrowinset=0.6,arrowlength=0.5,arrowsize=0.5pt 2.1,nodesep=0pt,labelsep=2pt}
\pcline{->}(3,3)(0,0) \naput{{\scriptsize $ r$}}}}

{
\rput[c](33,20){\psset{unit=1mm,doubleline=true,arrowinset=0.6,arrowlength=0.5,arrowsize=0.5pt 2.1,nodesep=0pt,labelsep=2pt}
\pcline{-}(2,2)(0,0) \naput{{\scriptsize $$}}}}

\endpspicture}

\rput(-5,40){$=$}

\rput(25,15){\pspicture(50,44)

%

\rput(20,12){\rnode{a2}{$(a,b)$}}
\rput(46,12){\rnode{a4}{$(Fa,Fb)$}}

\rput(46,30){\rnode{c4}{\hs{2}$(Fa,Fb)(Fa,Fa)$}}
\rput(0,40){\rnode{d1}{$(a,b)1$}}
\rput(28,40){\rnode{d3}{$(Fa,Fb)1$}}

\ncline{->}{d1}{d3} \naput{{\scriptsize $F1$}} 
\ncline{->}{d1}{a2} \nbput{{\scriptsize $\sim$}}

\ncline{->}{d3}{c4} \naput{{\scriptsize $1I'$}} 
\ncline{->}{c4}{a4} \naput{{\scriptsize $m'$}}
\ncline{->}{d3}{a4} \nbput{{\scriptsize $\sim$}}

\ncline{->}{a2}{a4} \nbput[labelsep=3pt]{{\scriptsize $F$}}


{
\rput[c](41,23){\psset{unit=1mm,doubleline=true,arrowinset=0.6,arrowlength=0.5,arrowsize=0.5pt 2.1,nodesep=0pt,labelsep=1pt}
\pcline{->}(3,3)(0,0) \naput{{\scriptsize $r'$}}}}

\rput[c](23,27){{\scriptsize $=$}}

\eps}

\endpspicture\]

\item left unit

\[\psset{unit=0.12cm,labelsep=2pt,nodesep=3pt}
\pspicture(50,70)

\rput(25,55){\pspicture(0,8)(50,40)

%

\rput(20,12){\rnode{a2}{$(a,b)$}}
\rput(46,12){\rnode{a4}{$(Fa,Fb)$}}

\rput(20,30){\rnode{c2}{$(b,b)(a,b)$}}
\rput(46,30){\rnode{c4}{$(Fb,Fb)(Fa,Fb)$}}
\rput(0,40){\rnode{d1}{$1(a,b)$}}
\rput(28,40){\rnode{d3}{$1(Fa,Fb)$}}

\ncline{->}{d1}{d3} \naput{{\scriptsize $1F$}} 
\ncline{->}{d3}{c4} \naput{{\scriptsize $I'1$}}
\ncline{->}{d1}{c2} \naput{{\scriptsize $I1$}}
\ncline{->}{c2}{c4} \naput{{\scriptsize $FF$}}

\ncline{->}{d1}{a2} \nbput{{\scriptsize $\sim$}}
\ncline{->}{c2}{a2} \naput{{\scriptsize $m$}}

\ncline{->}{c4}{a4} \naput{{\scriptsize $m'$}} 
\ncline{->}{a2}{a4} \nbput[labelsep=3pt]{{\scriptsize $F$}}


{
\rput[c](25,34){\psset{unit=1mm,doubleline=true,arrowinset=0.6,arrowlength=0.5,arrowsize=0.5pt 2.1,nodesep=0pt,labelsep=1pt}
\pcline{-}(2,2)(0,0) \naput{{\scriptsize $$}}}}

{
\rput[c](14,24){\psset{unit=1mm,doubleline=true,arrowinset=0.6,arrowlength=0.5,arrowsize=0.5pt 2.1,nodesep=0pt,labelsep=2pt}
\pcline{->}(3,3)(0,0) \naput{{\scriptsize $ l$}}}}

{
\rput[c](33,20){\psset{unit=1mm,doubleline=true,arrowinset=0.6,arrowlength=0.5,arrowsize=0.5pt 2.1,nodesep=0pt,labelsep=2pt}
\pcline{-}(2,2)(0,0) \naput{{\scriptsize $$}}}}

\endpspicture}

\rput(-5,40){$=$}

\rput(25,15){\pspicture(50,44)

%

\rput(20,12){\rnode{a2}{$(a,b)$}}
\rput(46,12){\rnode{a4}{$(Fa,Fb)$}}

\rput(46,30){\rnode{c4}{\hs{2}$(Fb,Fb)(Fa,Fb)$}}
\rput(0,40){\rnode{d1}{$1(a,b)$}}
\rput(28,40){\rnode{d3}{$1(Fa,Fb)$}}

\ncline{->}{d1}{d3} \naput{{\scriptsize $1F$}} 
\ncline{->}{d1}{a2} \nbput{{\scriptsize $\sim$}}

\ncline{->}{d3}{c4} \naput{{\scriptsize $I'1$}} 
\ncline{->}{c4}{a4} \naput{{\scriptsize $m'$}}
\ncline{->}{d3}{a4} \nbput{{\scriptsize $\sim$}}

\ncline{->}{a2}{a4} \nbput[labelsep=3pt]{{\scriptsize $F$}}


{
\rput[c](41,23){\psset{unit=1mm,doubleline=true,arrowinset=0.6,arrowlength=0.5,arrowsize=0.5pt 2.1,nodesep=0pt,labelsep=1pt}
\pcline{->}(3,3)(0,0) \naput{{\scriptsize $l'$}}}}

\rput[c](23,27){{\scriptsize $=$}}

\eps}

\endpspicture\]

\end{itemize}

\end{definition}

\begin{definition}[$K$-icons]

Let $X,Y$ be weak $K$-categories, and let $F,G \: X \tra Y$ be strict $K$-functors such that $Fa=Ga$ for all objects $a \in X$.  A \emph{$K$-icon}
\[
\psset{unit=0.1cm,labelsep=2pt,nodesep=2pt}
\pspicture(20,20)

\rput(0,10){\rnode{a1}{$X$}}  
\rput(20,10){\rnode{a2}{$Y$}}  

\ncarc[arcangle=45]{->}{a1}{a2}\naput{{\scriptsize $F$}}
\ncarc[arcangle=-45]{->}{a1}{a2}\nbput{{\scriptsize $G$}}

\pcline[linewidth=0.6pt,doubleline=true,arrowinset=0.6,arrowlength=0.8,arrowsize=0.5pt 2.1]{->}(10,13)(10,7)  \naput{{\scriptsize $\sigma$}}


\endpspicture
\]
is given by, for all pairs of objects $a,b \in X$ a 2-cell

\[
\psset{unit=0.1cm,labelsep=2pt,nodesep=2pt}
\pspicture(0,-4)(30,24)

\rput(0,10){\rnode{a1}{$X(a,b)$}}  
\rput(30,13){\rnode{a2}{$Y(Fa,Fb)$}}  
\rput(31,10){\rotatebox{90}{$=$}}
\rput(30,7){\rnode{a3}{$Y(Ga,Gb)$}}  

\ncarc[arcangle=45]{->}{a1}{a2}\naput[npos=0.6]{{\scriptsize $F$}}
\ncarc[arcangle=-45]{->}{a1}{a3}\nbput[npos=0.6]{{\scriptsize $G$}}

{
\rput[c](17,8.5){\psset{unit=1mm,doubleline=true,arrowinset=0.6,arrowlength=0.5,arrowsize=0.5pt 2.1,nodesep=0pt,labelsep=2pt}
\pcline{->}(0,3)(0,0) \nbput{{\scriptsize $\sigma_{ab}$}}}}


\endpspicture
\]
satisfying the following axioms:

\begin{itemize}
 \item Composition:

\[
\psset{unit=0.09cm,labelsep=2pt,nodesep=2pt}
\pspicture(0,-15)(80,40)

\rput(-5,15){
\pspicture(30,35)

%

\rput(0,30){\rnode{a1}{$(b,c)(a,b)$}}  
\rput(30,33){\rnode{a2}{\makebox[4em][l]{$(Fb,Fc)(Fa,Fb)$}}}  
\rput(38,30){\rotatebox{90}{$=$}}
\rput(30,27){\rnode{a3}{\makebox[4em][l]{$(Gb,Gc)(Ga,Gb)$}}}  

\rput(0,3){\rnode{b1}{$(a,c)$}}
\rput(30,3){\rnode{b2}{\makebox[3em][l]{$(Ga,Gc)$}}}

\ncarc[arcangle=45]{->}{a1}{a2}\naput[npos=0.56]{{\scriptsize $FF$}}
\ncarc[arcangle=-45]{->}{a1}{a3}\nbput[npos=0.56]{{\scriptsize $GG$}}

\ncline{->}{a1}{b1} \nbput{{\scriptsize $m$}}
\ncline{->}{a3}{b2} \naput{{\scriptsize $m'$}}

\ncarc[arcangle=-45]{->}{b1}{b2}\nbput[npos=0.5]{{\scriptsize $G$}}

{
\rput[c](17,28.5){\psset{unit=1mm,doubleline=true,arrowinset=0.6,arrowlength=0.5,arrowsize=0.5pt 2.1,nodesep=0pt,labelsep=2pt}
\pcline{->}(0,3)(0,0) \nbput{{\scriptsize $\sigma\sigma$}}}}

{
\rput[c](17,8.5){\psset{unit=1mm,doubleline=true,arrowinset=0.6,arrowlength=0.5,arrowsize=0.5pt 2.1,nodesep=0pt,labelsep=2pt}
\pcline{-}(0,2)(0,0) \nbput{{\scriptsize $$}}}}


\rput(50,15){$=$}

\endpspicture}

\rput(60,15){\pspicture(30,35)

%
%

\rput(0,30){\rnode{a1}{$(b,c)(a,b)$}}  
\rput(30,30){\rnode{a2}{\makebox[4em][l]{$(Fb,Fc)(Fa,Fb)$}}}  

\rput(0,3){\rnode{b1}{$(a,c)$}}
\rput(30,6){\rnode{b2}{\makebox[3em][l]{$(Fa,Fc)$}}}

\rput(32,3){\rotatebox{90}{$=$}}
\rput(30,0){\rnode{b3}{\makebox[3em][l]{$(Ga,Gc)$}}}  

\ncarc[arcangle=45]{->}{a1}{a2}\naput[npos=0.56]{{\scriptsize $FF$}}

\ncline{->}{a1}{b1} \nbput{{\scriptsize $m$}}
\ncline{->}{a2}{b2} \naput{{\scriptsize $m'$}}

\ncarc[arcangle=45]{->}{b1}{b2}\naput[npos=0.56]{{\scriptsize $F$}}
\ncarc[arcangle=-45]{->}{b1}{b3}\nbput[npos=0.56]{{\scriptsize $G$}}

{
\rput[c](16,21){\psset{unit=1mm,doubleline=true,arrowinset=0.6,arrowlength=0.5,arrowsize=0.5pt 2.1,nodesep=0pt,labelsep=2pt}
\pcline{-}(0,2)(0,0) \nbput{{\scriptsize $$}}}}

{
\rput[c](17,0.5){\psset{unit=1mm,doubleline=true,arrowinset=0.6,arrowlength=0.5,arrowsize=0.5pt 2.1,nodesep=0pt,labelsep=2pt}
\pcline{->}(0,3)(0,0) \nbput{{\scriptsize $\sigma$}}}}


\eps}

\endpspicture
\]

\item Unit:

\[
\psset{unit=0.14cm,labelsep=2pt,nodesep=2pt}
\pspicture(70,20)

\rput(15,10){\pspicture(45,20)


\rput(2,10){\rnode{a1}{$1$}}
\rput(25,18){\rnode{a2}{$(a,a)$}}
\rput(25,2){\rnode{a3}{$(Fa,Fa)=(Ga,Ga)$}}

\ncline{->}{a1}{a2} \naput{{\scriptsize $I$}}
\ncline{->}{a1}{a3} \nbput{{\scriptsize $I'$}}

\ncarc[arcangle=45]{->}{a2}{a3}\naput[npos=0.5]{{\scriptsize $G$}}
\ncarc[arcangle=-45]{->}{a2}{a3}\nbput[npos=0.5]{{\scriptsize $F$}}

{
\rput[c](14,9){\psset{unit=1mm,doubleline=true,arrowinset=0.6,arrowlength=0.5,arrowsize=0.5pt 2.1,nodesep=0pt,labelsep=-2pt}
\pcline{-}(0,0)(2,2) \naput{{\scriptsize $$}}}}

{
\rput[c](24,10){\psset{unit=1mm,doubleline=true,arrowinset=0.6,arrowlength=0.5,arrowsize=0.5pt 2.1,nodesep=0pt,labelsep=2pt}
\pcline{->}(0,0)(4,0) \nbput{{\scriptsize $\sigma$}}}}

\rput(40,10){$=$}

\endpspicture}

\rput(50,10){\pspicture(30,20)


\rput(2,10){\rnode{a1}{$1$}}
\rput(25,18){\rnode{a2}{$(a,a)$}}
\rput(25,2){\rnode{a3}{$(Fa,Fa)=(Ga,Ga)$}}

\ncline{->}{a1}{a2} \naput{{\scriptsize $I$}}
\ncline{->}{a1}{a3} \nbput{{\scriptsize $I'$}}

\ncarc[arcangle=45]{->}{a2}{a3}\naput[npos=0.5]{{\scriptsize $G$}}

{
\rput[c](17,9){\psset{unit=1mm,doubleline=true,arrowinset=0.6,arrowlength=0.5,arrowsize=0.5pt 2.1,nodesep=0pt,labelsep=-2pt}
\pcline{-}(0,0)(2,2) \naput{{\scriptsize $$}}}}


\eps}

\endpspicture\]

\end{itemize}

\end{definition}

In \cite[Prop. 3.7]{cg4} it is shown that for a 2-category $K$, weak $K$-categories, weak $K$-functors and $K$-icons form a 2-category; it follows that weak $K$-categories, \emph{strict} $K$-functors and $K$-icons also form a 2-category.

\begin{definition}\label{kwcat}
Given a 2-category $K$, write \cat{$K$-wCat} for the 2-category of weak $K$-categories, strict $K$-functors, and $K$-icons.
\end{definition}

Finally, we need the following result to ensure that we can iterate the weak enrichment construction.

\begin{proposition} \label{kprods}
Let $K$ be a 2-category with products. Then \cat{$K$-wCat} has products.
\end{proposition}

\begin{proof}
We know from \cite[Theorem 3.8]{cg4} that if $K$ is a symmetric monoidal 2-category we can construct a tensor product of weak $K$-categories. In our case, the symmetric monoidal structure on $K$ is (strict) products, and it is routine to check that the tensor product induced on weak $K$-categories inherits the universal property of a strict product. 
\end{proof}

\noi Note that eventually this notation will become a little unwieldy and we will write $\Cat(K)$ for $\cat{$K$-wCat}$; while this may seem misleading in a broader context, it is unambiguous in this work as we are never working with internal categories or strict enrichment.


\subsection{Background on $n$-monoidal categories}

Our environment for weak \eh\ is going to be categories equipped with $n$ monoidal structures in a suitably coherent way. The literature on this is somewhat piecemeal, so we think it will be helpful to make some details explicit.

The various treatments of this topic include different levels of strictness, and different levels of abstraction. The definition can be made directly via a long list of axioms, or abstractly via iterated internalisation. These concepts have been casually conflated in the literature but we find it helpful to keep them separate. On the one hand, we will define ``$n$-monoidal categories'' to be categories equipped with $n$ monoidal structures satisfying various axioms. On the other hand, we will define ``$n$-fold internal pseudomonoids'' by iterating a pseudomonoid construction, which we will in turn define via degenerate weak enrichment. Note that in this work we will be defining these structures directly, not as algebras for a monad, so the monoidal structures are all biased.

When iterating an internal monoid construction, the different levels of strictness in the literature essentially arise from

\begin{itemize}

\item The level of strictness of the internal monoids themselves; these are sometimes weak and sometimes strict.

\item The level of strictness of the morphisms between monoids, giving rise to interchange and some slightly technical unit conditions. In the literature the functoriality is usually lax, where ours will be strict, but sometimes the unit conditions for the lax morphisms are weak or even strict.

\end{itemize}

\noi Here is a summary of the levels used; each work has a specific use in mind and so uses the appropriate level of strictness for that:

\begin{itemize}

\item Balteanu, Fiedorowicz, Schw\"anzl, Vogt \cite{bfsv1}: strict monoids, lax morphisms with strict unit conditions

\item Forcey and Humes \cite{fh1}: weakly associative strictly unital monoids, lax morphisms with strict unit conditions

\item Aguiar and Mahajan \cite{am1}: weak monoids, lax morphisms with lax unit conditions (also Batanin and Markl \cite{bm1})

\item Garner and L\'opez Franco \cite{glf1} (normal duoidal categories): weak monoids, lax morphisms with weak unit conditions

\item present work: weak monoids, strict morphisms

\end{itemize}

\noi Thus the structures we need are substantially stricter than generality. In effect, the strict morphisms between weak monoids give rise to strict interchange between the different monoidal structures (and also some other technical conditions on units) and strict interchange in our eventual $n$-categories. A further level of strictness we will use is that we will always work in strict 2-categories, with symmetric monoidal structure always given by products.

%

Internal pseudomonoids can be defined directly, but we prefer to define them via degenerate enriched categories, as we eventually rely on that relationship.

\begin{definition}\label{dcatk}
Let $K$ be a cartesian 2-category. We define the 2-category $d\big(\Cat(K)\big)$ to be the full sub-2-category of $\Cat(K)$ whose 0-cells are the degenerate $K$-categories, that is, those whose set of objects is terminal.
\end{definition}

\begin{definition}\label{psmon}
A \emph{pseudomonoid in $K$} is a degenerate weak $K$-category where we have ignored the trivial set of objects; that is, it is a 0-cell of $K$ equipped with 1-cells for unit and multiplication, 2-cells for constraints, and the axioms as in Definition~\ref{wencat}. We write $\Mon(K)$ for the 2-category $d\big(\Cat(K)\big)$ regarded in this way; the 0-cells are pseudomonoids in $K$, the 1-cells are strict maps, and the 2-cells are transformations.
\end{definition}

We could more expressively write this 2-category as $\cat{PsMon}{(K)}_s$, but here again the notation will become unwieldy so we have decided to use more streamlined notation; as we will never be considering strict internal monoids or weak maps in this work this will not be ambiguous.

Pseudomonoids in $K$ are defined directly, without reference to degenerate weak $K$-categories, in \cite{mcc2} and \cite{am1}. Those definitions are weaker than we need, but invoking the appropriate extra strictness those definitions are the same as our definition via weak $K$-categories.

The level of strictness can be exemplified by our key example of $K = \Cat$.

\begin{example}
$\Mon(\Cat) = \cat{MonCat}_s$, the 2-category of weak monoidal categories, strict monoidal functors, and monoidal transformations.
\end{example}

The following result enables us to iterate the internal monoid construction.

\begin{proposition}
If $K$ has products then $\Mon(K)$ has products.
\end{proposition}

\begin{proof}
Given 0-cells $A, B \in \Mon(K)$ we know by Proposition~\ref{kprods} that they have a product in $\Cat(K)$, so we just need to show this is a product in $d\big(\Cat(K)\big)$. Now, the 0-cells of $A \times B$ as weak $K$-categories are given by $A_0 \times B_0$, so if $A_0 = B_0 = 1$ then $(A \times B)_0 = 1$, that is, $A \times B$ is degenerate. As we are working in a full sub-2-category of $\Cat(K)$, the universal property follows.
%
\end{proof}

We can now make the following iterated definition.

\begin{definition}
Let $K$ be a cartesian 2-category and $n \geq 1$. We write $\Mon^n(K)$ for the 2-category defined as follows:

\begin{itemize}

\item For $n=1$: $\Mon^1(K) := \Mon(K)$.

%

\item For $n > 1$: $\Mon^n(K) := \Mon\big(\Mon^{n-1}(K)\big)$.


\end{itemize}

\noi We call the 0-cells of $\Mon^n(K)$ \emph{$n$-fold internal monoids in $K$}.

\end{definition}

Unravelling the structure of a 2-fold internal monoid in $\Cat$ gives us the notion of a 2-monoidal category. In the literature it is fairly standard to conflate 2-monoidal categories and 2-fold internal monoids in \cat{Cat}. We find it helpful to distinguish between those equivalent structures.  The definition of 2-monoidal category can be read off \cite[Definition 6.1]{am1} with some extra strictness, but we prefer to find it by unravelling the definition of a 2-fold internal monoid in \Cat, to see where the axioms come from. 

We will write out this definition in full, partly because it is moderately enlightening, and partly because we will need to refer to some of these commuting diagrams precisely, when we prove that the Eckmann--Hilton Sphere commutes. We have chosen to present the axioms in a table as it helps us to see how the dual axioms arise from non-dual parts of the structure.

A 2-fold internal monoid in \Cat\ is a pseudomonoid in $\cat{MonCat}_s$, the cartesian 2-category of weak monoidal categories, strict functors, and monoidal transformations. Thus it is given by

\begin{itemize}

\item a weak monoidal category $A$, equipped with

\item strict monoidal functors $1 \tmap{\eta} A \tlmap{\mu} A \times A$, and

\item monoidal transformations $\alpha, l, r$ for associativity and left and right unit constraints

\end{itemize}

\noi satisfying the associativity and unit axioms.  We write the underlying monoidal structure of $A$ as $\big(A, (\vtp{\white{1}}{\white{1}}), V, \alpha, \tau, \beta\big)$, so the tensor product of $a$ and $b$ is \hs{0.0} $\vtp{a}{b}$ \hs{0.0} and the unit constraints are:
\[\vtp{a}{V} \hs{0.4} \tmap{\beta} \hs{0.2}  a, \hs{2} \vtp{V}{a}  \hs{0.4} \tmap{\tau} \hs{0.2}  a.\]

\noi (Note that we will write $\alpha$ for both the vertical and horizontal associativity constraints as there is no ambiguity.) Then the underlying functors, transformations, and axioms for $\eta, \mu, \alpha, l, r$ give another monoidal structure on $A$, which we write as $\big(A, (\htp{}{}), H, \alpha, l, r\big)$, so $\mu(a, b) = \htp{a}{b}$ and the unit constraints are

\[\htp{a}{H} \hs{0.0} \tmap{r} \hs{0.2}  a, \hs{2} \htp{H}{a}  \hs{0.0} \tmap{l} \hs{0.2}  a.\]

\noi So we know that we at least have two monoidal structures on $A$. We also have extra compatibility axioms coming from the monoidal conditions on $\eta, \mu, \alpha, l, r$. Note that a monoidal functor has 2 constraints, which we will call functoriality and unit (which are both identities in our case of strict monoidal functors), and 3 axioms, which we will call associativity, pre-unit, and post-unit (usually left and right). A monoidal transformation has 2 axioms, which we will call associativity and unit. We will write in some of the identities and even give some of them names as we find this helps with clarity.

\begin{proposition}\label{twofoldaxioms}
A 2-fold internal pseudomonoid in the 2-category $\cat{Cat}$ is equivalently a category equipped with two monoidal structures as above, with the following additional structure and axioms:

\[\pspicture(30,0)(160,190)
\begin{small}


\rput(50,201){\sf monoidal functor data}

\rput(50,182){\sf functoriality constraint}
\rput(50,167){\sf unit constraint}
\rput(50,120){\sf associativity axiom}
\rput(50,59){\sf pre-unit axiom}
\rput(50,18){\sf post-unit axiom}


\rput(105,0){\psline[linecolor=gray!60!white](0,0)(0,210)}
\rput(68,0){\psline[linecolor=gray!60!white](0,0)(0,210)}

\rput(0,191){\psline[linecolor=gray!60!white](30,0)(160,0)}
\rput(0,173){\psline[linecolor=gray!60!white](30,0)(160,0)}
\rput(0,161){\psline[linecolor=gray!60!white](30,0)(160,0)}
\rput(0,80){\psline[linecolor=gray!60!white](30,0)(160,0)}
\rput(0,40){\psline[linecolor=gray!60!white](30,0)(160,0)}


\rput(87,200){
\pspicture(-10,0)(10,14)

\rput(-10,10){\rnode{a1}{$1$}}  
\rput(10,10){\rnode{a2}{$A$}}  

\rput(-10,3){\rnode{b1}{$\ast$}}  
\rput(10,3){\rnode{b2}{$H$}}  

\ncline{->}{a1}{a2} \naput{{\scriptsize $\eta$}} 
\ncline{|->}{b1}{b2} \naput{{\scriptsize $$}} 

\endpspicture
}

\rput(87,182){
\pspicture(-10,-14)(10,14)

\rput(-10,0){\rnode{a1}{$\vtp{H}{H}$}}  
\rput(10,0){\rnode{a2}{$H$}}  

\ncline[nodesep=4pt]{->}{a1}{a2} \naput{\scr $\phi$} \nbput{{\scriptsize $=$}}  

\endpspicture
}

\rput(87,167){
\pspicture(-10,-14)(10,14)

\rput(-10,0){\rnode{a1}{$V$}}  
\rput(10,0){\rnode{a2}{$H$}}  

\ncline{->}{a1}{a2} \naput{{\scriptsize $=$}} 

\endpspicture
}


\rput(87,120){\pspicture(-10,-14)(10,14)



\rput(-10,20){\rnode{a1}{$\vtp{\vtp{H}{H}}{\hs{0.3}H\hs{0.3}}$}}  
\rput(10,20){\rnode{a2}{$\vtp{\hs{0.3}H\hs{0.3}}{\vtp{H}{H}}$}}  
\rput(-10,0){\rnode{a3}{$\vtp{H}{H}$}}  
\rput(10,0){\rnode{a4}{$\vtp{H}{H}$}}  
\rput(-10,-20){\rnode{a5}{$H$}}  
\rput(10,-20){\rnode{a6}{$H$}}  

\ncline{->}{a1}{a2} \naput{{\scriptsize $\alpha$}} 
\ncline{->}{a5}{a6} \nbput{{\scriptsize $=$}} 

\ncline{->}{a1}{a3} \nbput{{\scriptsize $=$}} 
\ncline{->}{a3}{a5} \nbput{{\scriptsize $=$}} 

\ncline{->}{a2}{a4} \naput{{\scriptsize $=$}} 
\ncline{->}{a4}{a6} \naput{{\scriptsize $=$}} 

\endpspicture}


\rput(87,59){\pspicture(-10,-14)(10,14)



\rput(-10,15){\rnode{a1}{$\vtp{H}{V}$}}  
\rput(10,15){\rnode{a2}{$H$}}  
\rput(-10,0){\rnode{a3}{$\vtp{H}{H}$}}  
\rput(-10,-15){\rnode{a5}{$H$}}  

\ncline{->}{a1}{a2} \naput{{\scriptsize $\beta_H$}} 
\ncline{->}{a5}{a2} \nbput{{\scriptsize $=$}} 

\ncline[nodesepA=0pt]{->}{a1}{a3} \nbput{{\scriptsize $=$}} 
\ncline{->}{a3}{a5} \nbput{{\scriptsize $=$}} 


\endpspicture}


\rput(87,18){\pspicture(-10,-14)(10,14)



\rput(-10,15){\rnode{a1}{$\vtp{V}{H}$}}  
\rput(10,15){\rnode{a2}{$H$}}  
\rput(-10,0){\rnode{a3}{$\vtp{H}{H}$}}  
\rput(-10,-15){\rnode{a5}{$H$}}  

\ncline{->}{a1}{a2} \naput{{\scriptsize $\tau_H$}} 
\ncline{->}{a5}{a2} \nbput{{\scriptsize $=$}} 

\ncline[nodesepA=0pt]{->}{a1}{a3} \nbput{{\scriptsize $=$}} 
\ncline{->}{a3}{a5} \nbput{{\scriptsize $=$}} 


\endpspicture}


\rput(132,200){
\pspicture(-10,0)(10,14)

\rput(-10,10){\rnode{a1}{$A \times A$}}  
\rput(10,10){\rnode{a2}{$A$}}  

\rput(-10,4){\rnode{b1}{$a, b$}}  
\rput(10,4){\rnode{b2}{$\htp{a}{b}$}}  

\ncline{->}{a1}{a2} \naput{{\scriptsize $\mu$}} 
\ncline{|->}{b1}{b2} \naput{{\scriptsize $$}} 

\endpspicture
}

\rput(132,182){
\pspicture(-10,-14)(10,14)

\rput(-10,0){\rnode{a1}{$\vvtp{\htp{\a}{\b}}{\htp{\c}{\d}}$}}  
\rput(10,0){\rnode{a2}{$\htp{\vtp{\a}{\c}}{\vtp{\b}{\d}}$}}  

\ncline[nodesep=4pt]{->}{a1}{a2} \naput{{\scriptsize $\chi$}} \nbput{\scr $=$} 

\endpspicture
}

\rput(132,167){
\pspicture(-10,-14)(10,14)

\rput(-10,0){\rnode{a1}{$V$}}  
\rput(10,0){\rnode{a2}{$\htp{V}{V}$}}  

\ncline{->}{a1}{a2} \naput{\scr $\phi$} \nbput{{\scriptsize $=$}} 

\endpspicture
}


\rput(132,120){\psset{unit=0.1cm,labelsep=3pt,nodesep=3pt}
\pspicture(-10,-24)(10,24)



\rput(-12,28){\rnode{tl}{$\vtp{\vtp{\htp{\a}{\b}\vs{0.2}}{\htp{\c}{\d}\vs{0.2}}}{\hs{0.3}\htp{\e}{\f}\hs{0.3}}$}}  

\rput(12,28){\rnode{tr}{$\vtp{\hs{0.3}\htp{\a}{\b}\vs{0.2}\hs{0.3}}{\vtp{\raisebox{-0.2em}{\htp{\c}{\d\vs{0.1}}}}{\htp{\e}{\f}}}$}}  

\rput(-12,0){\rnode{ml}{$\vtp{\htp{\vtp{\a}{\c}}{\vtp{\b}{\d}\vs{0.2}}}{\hs{0.3}\htp{\e}{\f}\hs{0.3}}$}}  

\rput(12,0){\rnode{mr}{$\vtp{\hs{0.3}\htp{\a}{\b}\hs{0.3}\vs{0.2}}{\raisebox{-0.7em}{\htp{\vtp{\c}{\e}}{\vtp{\d}{\f}}}}$}}  

\rput(-12,-28){\rnode{bl}{$\htp{\vtp{\vtp{\a}{\c}}{\hs{0.3}\e\hs{0.3}}}{\vtp{\vtp{\b}{\d}}{\hs{0.3}\f\hs{0.3}}}$}}  

\rput(12,-28){\rnode{br}{$\htp{\vtp{\hs{0.3}\a\hs{0.3}}{\vtp{\c}{\e}}}{\vtp{\hs{0.3}\b\hs{0.3}}{\vtp{\d}{\f}}}$}}  

\psset{arrows=->}

\ncline{tl}{tr} \naput{{\scriptsize $\alpha$}} 
\ncline{bl}{br} \nbput{{\scriptsize $\htp{\alpha}{\alpha}$}} 
\ncline{tl}{ml} \nbput{{\scriptsize $\vtp{\chi}{1}$}} \naput{\scr $=$} 
\ncline{ml}{bl} \nbput{{\scriptsize $\chi$}} \naput{\scr $=$} 
\ncline{tr}{mr} \naput{{\scriptsize $\vtp{1}{\chi}$}} \nbput{\scr $=$} 
\ncline{mr}{br} \naput{{\scriptsize $\chi$}} \nbput{\scr $=$} 

\eps}


\rput(132,59){
\pspicture(-22,-12)(22,12)

\psset{nodesep=4pt}


\rput[B](-19,11){\Rnode{a1}{$\vvtp{\htp{\a}{\b}}{V}$}}  
\rput[B](0,11){\Rnode{a2}{$\vvtp{\htp{\a}{\b}}{\htp{V}{V}}$}}  
\rput[B](19,11){\Rnode{a3}{$\htp{\vtp{\a}{V}}{\vtp{\b}{V}}$}}  

\rput[B](-19,-11){\Rnode{b1}{$\htp{\a}{\b}$}}  
\rput[B](19,-11){\Rnode{b3}{$\htp{\a}{\b}$}}  

\ncline{->}{a1}{a2} \naput[labelsep=0pt]{{\scalebox{0.8}{$\vtp{1}{\phi}$}}} \nbput{$=$} 
\ncline{->}{b1}{b3} \nbput{{\scriptsize $=$}} 
\ncline{->}{a2}{a3} \naput[npos=0.5]{\scr $\chi$} \nbput{\scr $=$} 
%
%
\ncline{->}{a1}{b1} \nbput{{\scriptsize $\beta$}} 
\ncline{->}{a3}{b3} \naput{{\scriptsize $\htp{\beta}{\beta}$}} 

\rput(0,0){\sf key axiom $\chi$--$\beta$}

\endpspicture}



\rput(132,18){
\pspicture(-22,-12)(22,12)

\psset{nodesep=4pt}


\rput[B](-19,11){\Rnode{a1}{$\vvtp{V}{\htp{\a}{\b}}$}}  
\rput[B](0,11){\Rnode{a2}{$\vvtp{\htp{V}{V}}{\htp{\a}{\b}}$}}  
\rput[B](19,11){\Rnode{a3}{$\htp{\vtp{V}{\a}}{\vtp{V}{\b}}$}}  

\rput[B](-19,-11){\Rnode{b1}{$\htp{\a}{\b}$}}  
\rput[B](19,-11){\Rnode{b3}{$\htp{\a}{\b}$}}  

\ncline{->}{a1}{a2} \naput[labelsep=0pt]{{\scalebox{0.8}{$\vtp{\phi}{1}$}}} \nbput{$=$} 
\ncline{->}{b1}{b3} \nbput{{\scriptsize $=$}} 
\ncline{->}{a2}{a3} \naput[npos=0.5]{\scr $\chi$} \nbput{\scr $=$} 
%
%
\ncline{->}{a1}{b1} \nbput{{\scriptsize $\tau$}} 
\ncline{->}{a3}{b3} \naput{{\scriptsize $\htp{\tau}{\tau}$}} 

\rput(0,0){\sf key axiom $\chi$--$\tau$}

\endpspicture}


\end{small}
\eps\]



\[\pspicture(20,10)(160,130)
\begin{small}


\rput(25,0){\psline[linecolor=gray!60!white](0,10)(0,130)}
\rput(79,0){\psline[linecolor=gray!60!white](0,10)(0,130)}
\rput(121,0){\psline[linecolor=gray!60!white](0,10)(0,130)}

\rput(0,92){\psline[linecolor=gray!60!white](0,0)(160,0)}
\rput(0,49){\psline[linecolor=gray!60!white](0,0)(160,0)}

\rput(12,110){\parbox{10em}{\sf\bc monoidal\\ transformation\\ data \ec}}
\rput(12,70){\parbox{10em}{\sf\bc associativity\\ axiom \ec}}
\rput(12,30){\parbox{10em}{\sf\bc unit\\ axiom \ec}}



\rput(52,110){
\pspicture(-10,-10)(10,10)


\rput(-10,10){\rnode{a1}{$A^3$}} 
\rput(10,10){\rnode{a2}{$A^2$}} 

\rput(-10,-10){\rnode{b1}{$A^2$}}   
\rput(10,-10){\rnode{b2}{$A$}}  

\psset{nodesep=3pt,labelsep=2pt,arrows=->}
\ncline{a1}{a2}\naput{{\scriptsize $A\mu$}} 
\ncline{b1}{b2}\nbput{{\scriptsize $\mu$}} 
\ncline{a1}{b1}\nbput{{\scriptsize $\mu A$}} 
\ncline{a2}{b2}\naput{{\scriptsize $\mu$}} 

\rput(0,0){
\psset{labelsep=1.5pt}
\pnode(2,2){a3}
\pnode(-2,-2){b3}
\ncline[doubleline=true,arrowinset=0.6,arrowlength=0.8,arrowsize=0.5pt 2.1,nodesep=0pt]{a3}{b3} \nbput[npos=0.4]{{\scriptsize $\alpha$}}
}

\endpspicture}


\rput(100,110){
\pspicture(-10,-10)(10,10)


\rput(-10,10){\rnode{a1}{$A$}} 
\rput(10,10){\rnode{a2}{$A^2$}} 

\rput(10,-10){\rnode{b2}{$A$}}  

\psset{nodesep=3pt,labelsep=2pt,arrows=->}
\ncline{a1}{a2}\naput{{\scriptsize $A\eta$}} 
\ncline{a1}{b2}\nbput{{\scriptsize $1$}} 
\ncline{a2}{b2}\naput{{\scriptsize $\mu$}} 

\rput(4,3){
\psset{labelsep=1.5pt}
\pnode(2,2){a3}
\pnode(-2,-2){b3}
\ncline[doubleline=true,arrowinset=0.6,arrowlength=0.8,arrowsize=0.5pt 2.1,nodesep=0pt]{a3}{b3} \nbput[npos=0.4]{{\scriptsize $l$}}
}

\endpspicture}


\rput(141,110){
\pspicture(-10,-10)(10,10)


\rput(-10,10){\rnode{a1}{$A$}} 
\rput(10,10){\rnode{a2}{$A^2$}} 

\rput(10,-10){\rnode{b2}{$A$}}  

\psset{nodesep=3pt,labelsep=2pt,arrows=->}
\ncline{a1}{a2}\naput{{\scriptsize $\eta A$}} 
\ncline{a1}{b2}\nbput{{\scriptsize $1$}} 
\ncline{a2}{b2}\naput{{\scriptsize $\mu$}} 

\rput(4,3){
\psset{labelsep=1.5pt}
\pnode(2,2){a3}
\pnode(-2,-2){b3}
\ncline[doubleline=true,arrowinset=0.6,arrowlength=0.8,arrowsize=0.5pt 2.1,nodesep=0pt]{a3}{b3} \nbput[npos=0.4]{{\scriptsize $r$}}
}

\endpspicture}


\rput(52,70){\psset{unit=0.1cm,labelsep=3pt,nodesep=3pt}
\pspicture(-10,-12)(10,12)



\rput[B](-19,11){\Rnode{a1}{$\vtp{\hhtp{\htp{\a}{\b}}{\c}\vs{0.2}}
{\raisebox{-0.2em}{\hhtp{\htp{\d}{\e}}{\f}}}$}}  

\rput[B](0,11){\Rnode{a2}{$\hhtp{\vtp{\htp{\a}{\b}}{\htp{\d}{\e}}}{\vtp{\c}{\f}}$}}  

\rput[B](19,11){\Rnode{a3}{$\hhtp{\htp{\vtp{\a}{\d}}{\vtp{\b}{\e}}}{\vtp{\c}{\vs{0.1}\raisebox{0.13em}{$\f$}}}$}}  

\rput[B](-19,-11){\Rnode{b1}{$\vtp{\hhtp{\a}{\htp{\b}{\c}}\vs{0.2}}{\raisebox{-0.2em}{\hhtp{\d}{\htp{\e}{\f}}}}$}}  

\rput[B](0,-11){\Rnode{b2}{$\hhtp{\vtp{\a}{\d}}{\vtp{\htp{\b}{\c}\vs{0.05}}{\htp{\e}{\f}}}$}}  

\rput[B](19,-11){\Rnode{b3}{$\hhtp{\vtp{\a}{\vs{0.1}\raisebox{0.14em}{$\d$}}}{\htp{\vtp{\b}{\e}}{\vtp{\c}{\f}}}$}}  


\ncline{->}{a1}{a2} \naput{{\scriptsize $\chi$}} \nbput{\scr $=$} 
\ncline{->}{b1}{b2} \nbput{{\scriptsize $\chi$}} \naput{\scr $=$} 

\ncline{->}{a2}{a3} \naput{{\scriptsize $\htp{\chi}{1}$}} \nbput{\scr $=$} 
\ncline{->}{b2}{b3} \nbput{{\scriptsize $\htp{1}{\chi}$}} \naput{\scr $=$} 

\ncline{->}{a1}{b1} \nbput{{\scriptsize $\vtp{\alpha}{\alpha}$}} 
\ncline{->}{a3}{b3} \naput{{\scriptsize $\alpha$}} 

\eps}


\rput(100,70){
\pspicture(-22,-12)(22,12)


\rput[B](-15,11){\Rnode{a1}{$\vvtp{\htp{H}{\a}}{\htp{H}{\b}}$}}  
\rput[B](0,11){\Rnode{a2}{$\htp{\vtp{H}{H}}{\vtp{\a}{\b}}$}}  
\rput[B](15,11){\Rnode{a3}{$\htp{H}{\vtp{\a}{\b}}$}}  

\rput[B](-15,-11){\Rnode{b1}{$\vtp{\a}{\b}$}}  
\rput[B](15,-11){\Rnode{b3}{$\vtp{\a}{\b}$}}  

\ncline{->}{a1}{a2} \naput[npos=0.45]{{\scriptsize $\chi$}} 
\ncline{->}{b1}{b3} \nbput{{\scriptsize $=$}} 

\ncline{->}{a2}{a3} \naput[npos=0.45]{{\scr {\htp{\phi}{1}}}} \nbput[npos=0.45]{\scr $=$} 

\ncline{->}{a1}{b1} \nbput{{\scriptsize $\vtp{l}{l}$}} 
\ncline{->}{a3}{b3} \naput{{\scriptsize $l$}} 

\rput(0,0){\sf key axiom $\chi$--$l$}

\endpspicture}



\rput(141,70){
\pspicture(-22,-12)(22,12)


\rput[B](-15,11){\Rnode{a1}{$\vvtp{\htp{\a}{H}}{\htp{\b}{H}}$}}  
\rput[B](0,11){\Rnode{a2}{$\htp{\vtp{\a}{\b}}{\vtp{H}{H}}$}}  
\rput[B](15,11){\Rnode{a3}{$\htp{\vtp{\a}{\b}}{H}$}}  

\rput[B](-15,-11){\Rnode{b1}{$\vtp{\a}{\b}$}}  
\rput[B](15,-11){\Rnode{b3}{$\vtp{\a}{\b}$}}  

\ncline{->}{a1}{a2} \naput[npos=0.45]{{\scriptsize $\chi$}} 
\ncline{->}{b1}{b3} \nbput{{\scriptsize $=$}} 

\ncline{->}{a2}{a3} \naput[npos=0.45]{{\scr {\htp{1}{\phi}}}} \nbput[npos=0.45]{\scr $=$} 

\ncline{->}{a1}{b1} \nbput{{\scriptsize $\vtp{r}{r}$}} 
\ncline{->}{a3}{b3} \naput{{\scriptsize $r$}} 

\rput(0,0){\sf key axiom $\chi$--$r$}

\endpspicture}


\rput(52,30){\pspicture(-10,-14)(10,14)



\rput(-12,10){\rnode{a1}{$V$}}  
\rput(12,10){\rnode{a2}{$\hhtp{V}{\htp{V}{V}}$}}  
\rput(12,-10){\rnode{a4}{$\hhtp{\htp{V}{V}}{V}$}}  

\ncline{->}{a1}{a2} \naput{{\scriptsize $=$}} 
\ncline{->}{a1}{a4} \nbput{{\scriptsize $=$}} 
\ncline{->}{a2}{a4} \naput{{\scriptsize $\alpha$}} 

\endpspicture
}


\rput(100,30){\psset{nodesep=4pt}
\pspicture(-10,-14)(10,14)



\rput(-12,10){\rnode{a1}{$V$}}  
\rput(12,10){\rnode{a2}{$\htp{H}{V}$}}  
\rput(12,-10){\rnode{a4}{$V$}}  

\ncline{->}{a1}{a2} \naput{{\scriptsize $=$}} 
\ncline{->}{a1}{a4} \nbput{{\scriptsize $=$}} 
\ncline{->}{a2}{a4} \naput{{\scriptsize $l_V$}} 

\endpspicture
}


\rput(141,30){\psset{nodesep=4pt}
\pspicture(-10,-14)(10,14)



\rput(-12,10){\rnode{a1}{$V$}}  
\rput(12,10){\rnode{a2}{$\htp{V}{H}$}}  
\rput(12,-10){\rnode{a4}{$V$}}  

\ncline{->}{a1}{a2} \naput{{\scriptsize $=$}} 
\ncline{->}{a1}{a4} \nbput{{\scriptsize $=$}} 
\ncline{->}{a2}{a4} \naput{{\scriptsize $r_V$}} 

\endpspicture
}

\end{small}
\eps\]

\end{proposition}

\begin{remarks}
Some key features of this level of strictness to note besides interchange being strict are:

\begin{enumerate}

\item $V = H$, that is, the horizontal and vertical units are the same, and we will generally write this unit as 1 except when we are emphasising that we are treating it as a horizontal or vertical unit.

\item The left-hand column in the first table and the bottom row of the second table show us that the unit 1 is strict on itself. (A priori $H$ is vertically strict on itself and $V$ is horizontally strict on itself, but $V$ and $H$ are the same.)

\item The interaction diagrams between $\chi$ and $\beta, \tau, l, r$ respectively turn out to be crucial, thus we have labelled them ``key axioms''.

\item The only remaining diagrams are the top right in the first table and top left on the second page, giving the interaction between $\chi$ and the vertical and horizontal associators. These will come into play when in future work when we show how to express these structures via distributive laws.


\end{enumerate}

\end{remarks}

Note that there is some asymmetry in these axioms that becomes moot as we are studying weak structures rather than lax ones. For the lax case, the second monoidal structure would be a lax monoid with respect to the first, and the first would become a colax monoid with respect to the second. However, for weak monoids, the second structure is a priori weak with respect to the first, and the first becomes a weak monoid with respect to the second. One consequence is that two monoidal structures end up playing analogous roles in the 2-monoidal structure, so it no longer matters which one is ``first'' and ``second''.

\begin{definition}\label{twomonoidal}
A \emph{2-monoidal category} (called semi-strict when emphasis is needed) is a category equipped with two monoidal structures satisfying the above compatibility axioms.
\end{definition}

Note that being a 2-monoidal category is sufficient but not necessary for weak \eh, just like for standard \eh\ having two monoid structures plus interchange is sufficient but not necessary since, for example, associativity can be deduced.


\begin{proposition}\label{threepsmon}
Let $C \in \Mon^3(\Cat)$. Then $C$ is equipped with $3$ monoidal structures, any two of which form an object of $\Mon^2(\Cat)$.
\end{proposition}

%

\begin{proof}

First note that $\Mon$ is a 2-functor $\cat{2-Cat}_c \tra \cat{2-Cat}_c$, where $\cat{2-Cat}_c$ denotes the 2-category of cartesian 2-categories, product-preserving functors, and cartesian monoidal transformations. Next note that for any cartesian 2-category $K$ we have a forgetful functor
\[U_K : \Mon(K) \tra K\]
giving the underlying object of each monoid in $K$, and it preserves products so is a 1-cell in $\cat{2-Cat}_c$. Now, certainly each application of $\Mon$ produces another monoidal structure, so $C$ certinly has three monoidal structures and we will write them as $\otimes_1, \otimes_2, \otimes_3$ in order. Then we have three forgetful functors
\[\Mon^3(K) \tra \Mon^2(K)\]
as follows:
\begin{itemize}
\item $\Mon^2( U_K)$, giving $\otimes_2, \otimes_3$
\item $\Mon (U_{\Mon(K)})$, giving $\otimes_1, \otimes_3$
\item $U_{\Mon^2(K)}$, giving $\otimes_1, \otimes_2$

\end{itemize}

\vs{-1.5}

\end{proof}

\begin{corollary}\label{ntothree}
Let $C \in \Mon^n(\Cat)$. Then we can apply the forgetful functor $U$ $n-3$ times to obtain an object of $\Mon^3(\Cat)$.
\end{corollary}

\section{Weak \eh\ on categories}\label{braided}

We will now give a brief recap and overview of the first generalisation of \eha\ to higher dimensions, in which sets with monoid structures are replaced by categories with weak monoidal structures.  Our framework will be a semi-strict 2-monoidal category as in Definition~\ref{twomonoidal}, or equivalently, a 2-fold internal pseudomonoid in \cat{Cat}. Thus we are in particular working with strict interchange (and some other strictness conditions).

This is not the most general case; our main reason for this is that eventually we will study degenerate semi-strict higher categories in which interchange is strict. This is conjectured to be ``weak enough'' to be equivalent to a fully weak version, and so we make use of the simplified combinatorics that the strict interchange yields. In a sense we are trying to do the opposite of a very general situation: instead of being as weak as possible, we want to show how strict the structure can be without collapsing to something non-equivalent.

%

We will start with 2-monoidal categories and show explicitly how to produce braided monoidal categories from them. A more general result for a weaker case was proved by Joyal and Street, with one monoidal structure and one ``multiplication'', which is essentially a monoidal structure without associativity constraints \emph{a priori}; associativity is in any case not involved with \eha. Furthermore, interchange is taken to be weak. In \cite{js1} the proof is more abstract, but in the earlier version of the paper \cite{js3} a commuting diagram is provided, which is essentially the one in our proof below. The main difference is that we write the first monoidal structure horizontally and the second one vertically, which gives some extra geometric insight into the situation. Our calculations are further helped by the extra strictness we are using; in \cite{cc1, cc2} we proved that the extra strictness does not cause a collapse of the structure to a symmetry.



\begin{theorem}[Weak Eckmann--Hilton argument]

Let $(C, (\ |\ ), \left(\frac{\ }{\ }\right), I)$ be a 2-monoidal category. Then the following is a braiding on $(C, \left(\frac{\ }{\ }\right), I)$:

%
%
%
%
%
%
%
%
%
%

\[\psset{unit=1mm}\pspicture(0,-5)(80,7)

\rput(-4,0){\rnode{a1}{\htp{a}{b}}}

\rput(20,0){\rnode{a2}{\fourtp{a}{}{}{b}}}

\rput(40,0){\rnode{a3}{\vtp{a}{b}}}

\rput(60,0){\rnode{a4}{\fourtp{}{a}{b}{}}}

\rput(82,0){\rnode{a5}{\htp{b}{a}}}

\psset{arrows=->, nodesep=7pt}

\ncline[nodesepA=3pt]{a1}{a2} \naput{\scr $\htp{\beta^{-1}}{\tau^{-1}}$}
\ncline{a2}{a3}  \naput{\scr $\vtp{r}{l}$}
\ncline{a3}{a4}  \naput{\scr $\vtp{l^{-1}}{r^{-1}}$}
\ncline[nodesepB=3pt]{a4}{a5}  \naput{\scr $\htp{\tau}{\beta}$}

\eps\]

\end{theorem}

\begin{proof}
We need to check two axioms for the braiding. One is as follows:

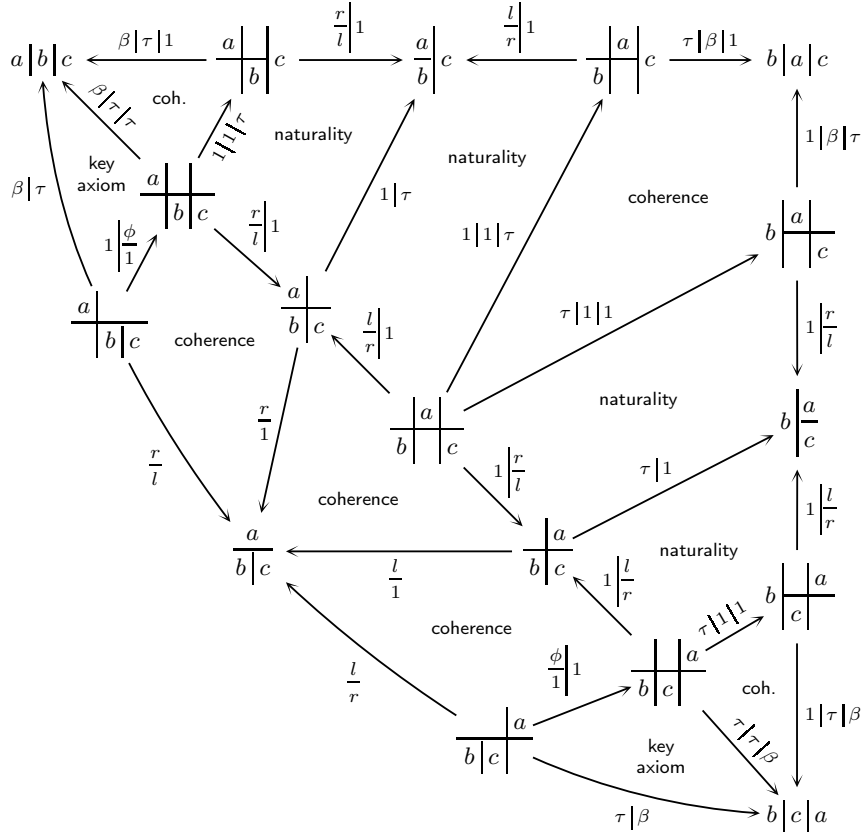
\[\begin{small}\ps(110,105)


\rput(0,100){\rnode{tl}{\small $\threehtp{a}{b}{c}$}}
\rput(28,100){\rnode{t2}{\small $\htp{\fourtp{a}{}{}{b}}{c}$}}
\rput(52,100){\rnode{t3}{\small $\htp{\vtp{a}{b}}{c}$}}
\rput(77,100){\rnode{t4}{\small $\htp{\fourtp{}{a}{b}{}}{c}$}}
\rput(100,100){\rnode{tr}{\small $\threehtp{b}{a}{c}$}}

\rput(100,77){\rnode{r2}{\small $\htp{b}{\fourtp{a}{}{}{c}}$}}
\rput(100,52){\rnode{r3}{\small $\htp{b}{\vtp{a}{c}}$}}
\rput(100,29){\rnode{r4}{\small $\htp{b}{\fourtp{}{a}{c}{}}$}}
\rput(100,0){\rnode{br}{\small $\threehtp{b}{c}{a}$}}

\rput(18,82){\rnode{m1}{\small $\sixtp{a}{}{}{}{b}{c}$}}
\rput(35,67){\rnode{m2}{\small $\fourtp{a}{}{b}{c}$}}
\rput(51,51){\rnode{m3}{\small $\sixtp{}{a}{}{b}{}{c}$}}
\rput(67,35){\rnode{m4}{\small $\fourtp{}{a}{b}{c}$}}
\rput(83,19){\rnode{m5}{\small $\sixtp{}{}{a}{b}{c}{}$}}

\rput(9,65){\rnode{l2}{\small $\fourtp{a}{}{}{\htp{b}{c}}$}}
\rput(28,35){\rnode{l3}{\small $\vtp{a}{\htp{b}{c}}$}}
\rput(60,10){\rnode{l4}{\small $\fourtp{}{a}{\htp{b}{c}}{}$}}

\psset{arrows=->, nodesep=5pt}

\ncline{t2}{tl} \nbput{\scr $\threehtpscr{\beta}{\tau}{1}$}
\ncline{t2}{t3} \naput{\scr $\htp{\vtp{r}{l}}{1}$}
\ncline{t4}{t3} \nbput{\scr $\htp{\vtp{l}{r}}{1}$}
\ncline{t4}{tr} \naput{\scr $\threehtpscr{\tau}{\beta}{1}$}
\ncline{r2}{tr} \nbput{\scr $\threehtpscr{1}{\beta}{\tau}$}
\ncline{r2}{r3} \naput{\scr $\htp{1}{\vtp{r}{l}}$}
\ncline{r4}{r3} \nbput{\scr $\htp{1}{\vtp{l}{r}}$}
\ncline{r4}{br} \naput{\scr $\threehtpscr{1}{\tau}{\beta}$}
\ncline[nodesep=2pt]{<-}{tl}{m1}  \naput[nrot=:U]{\scr $\threehtpscr{\beta}{\tau}{\tau}$}

\psset{nodesep=2pt}

\ncline[nodesep=2pt]{m1}{t2} \nbput[nrot=:U]{\scr $\threehtp{1}{1}{\tau}$}
\ncline[nodesepA=0pt,nodesepB=2pt]{l2}{m1} \naput[labelsep=0pt]{\scr $\htp{1}{\vtp{\phi}{1}}$}
\ncline[nodesep=2pt]{m1}{m2} \naput{\scr $\htp{\vtp{r}{l}}{1}$}
\ncline{m2}{t3} \nbput{\scr $\htp{1}{\tau}$}
\ncline{m2}{l3} \nbput{\scr $\vtp{r}{1}$}
\ncline{m3}{t4} \naput{\scr $\threehtp{1}{1}{\tau}$}
\ncline[nodesep=2pt]{m3}{m2} \nbput{\scr $\htp{\vtp{l}{r}}{1}$}
\ncline{m3}{m4} \naput[labelsep=-1pt]{\scr \htp{1}{\vtp{r}{l}}}
\ncline{m3}{r2} \naput[labelsep=0pt]{\scr \threehtp{\tau}{1}{1}  }
\ncline{m4}{r3} \naput[labelsep=0pt]{\scr \htp{\tau}{1}   }
\ncline[nodesep=6pt]{m4}{l3} \naput[labelsep=1pt]{\scr \vtp{l}{1}  }
\ncline{m5}{m4} \nbput[labelsep=-1pt]{\scr \htp{1}{\vtp{l}{r}}}
\ncline[nodesepB=0pt]{m5}{r4} \naput[nrot=:U]{\scr \threehtp{\tau}{1}{1}  }
\ncline{l4}{m5} \naput{\scr \htp{\vtp{\phi}{1}}{1}  }
\ncline{m5}{br} \naput[nrot=:U]{\scr \threehtpscr{\tau}{\tau}{\beta}}

\ncarc[arcangle=10]{l2}{tl} \naput[labelsep=0pt]{\scr \htp{\beta}{\tau}  }
\ncarc[arcangle=-5]{l2}{l3} \nbput{\scr \vtp{r}{l}  }
\ncarc[arcangle=5, nodesepB=6pt]{l4}{l3} \naput{\scr \vtp{l}{r}  }
\ncarc[arcangle=-10]{l4}{br} \nbput{\scr \htp{\tau}{\beta}   }

\rput(8,85){\parbox{5em}{\bc \scr\sf key\\ axiom \ec}}
\rput(82,8){\parbox{5em}{\bc \scr\sf key \\ axiom \ec}}

\rput(17,95){\scr\sf coh.}
\rput(95,17){\scr\sf coh.}

\rput(23,63){\scr\sf coherence}
\rput(42,42){\scr\sf coherence}
\rput(57,25){\scr\sf coherence}
\rput(83,82){\scr\sf coherence}

\rput(36,90){\scr\sf naturality}
\rput(59,87){\scr\sf naturality}
\rput(79,55){\scr\sf naturality}
\rput(87,35){\scr\sf naturality}

\eps\end{small}\]

\noi Some notes on this diagram:

\begin{itemize}

\item We ignore associativity as it is only used in the horizontal direction and so we can invoke coherence.

\item We include the directions on the arrows for clarity, but every arrow is an isomorphism (or even an equality) and some are used as inverses.

\item When indicating why individual portions of the diagram commute we write ``naturality'' for diagrams of the form

\[
\psset{unit=0.1cm,labelsep=3pt,nodesep=3pt}
\pspicture(0,-2)(20,17)



\rput[B](0,15){\Rnode{tl}{$$}}  
\rput[B](20,15){\Rnode{tr}{$$}}  
\rput[B](0,0){\Rnode{bl}{$$}}  
\rput[B](20,0){\Rnode{br}{$$}}  

\ncline{->}{tl}{tr} \naput{{\scriptsize $\htp{f}{1}$}} 
\ncline{->}{bl}{br} \nbput{{\scriptsize $\htp{f}{1}$}} 
\ncline{->}{tl}{bl} \nbput{{\scriptsize $\htp{1}{g}$}} 
\ncline{->}{tr}{br} \naput{{\scriptsize $\htp{1}{g}$}} 

\endpspicture
\]

\item The regions marked ``coherence'' indicate coherence for horizontal units; the top and bottom halves of each object and each morphism can be separated out and considered separately, invoking just the horizontal tensor product for each half.

\item The regions marked ``key axiom'' are referring to the axiom labelled ``key axiom $\chi$--$\tau$ in the list of axioms for a 2-fold internal pseudomonoid in \cat{Cat} (Proposition~\ref{twofoldaxioms}).

\end{itemize}

The other diagram follows similarly. \end{proof}

Note that there is a complete cycle (``clock'') as in the diagram in Section~\ref{backgroundeha}, and any half circle between ``9 o'clock'' and ``3 o'clock'' is a braiding, as is any half circle between 12 o'clock and 6 o'clock. Note that if we are constructing a braiding from a fully weak doubly-degenerate tricategory, then only vertical composition directly yields a monoidal structure --- horizontal composition might not yield a monoidal structure, because of weak 1-cell units. So in that more general framework we do not immediately get a 2-monoidal category, just a monoidal category with some extra structure, and \cite{js1} construct the braiding with that in mind. Having both of the compositions giving genuine monoidal structures simplifies this a great deal; it means we immediately get a 2-monoidal category and can use any of the monoidal structures as the principal one for the braiding. Geometrically it is more satisfying to use the horizontal one and go from 9 o'clock to 3 o'clock for the braiding, across the page instead of down the page. When we study 3-monoidal structures we can then put the third monoidal structure in a geometrically sensible third direction.


Note further that going clockwise from 9 o'clock to 3 o'clock is the same process as going clockwise from 3 o'clock to 9 o'clock, just with the elements labelled differently; however there is no reason for the anticlockwise direction to be the same. That is, there is no reason for this braiding to be a symmetry. We showed in \cite{cc2} that doubly-degenerate semi-strict tricategories do yield ``all'' possible braided monoidal categories (in a precise sense); in fact there we did it even more strictly, with one of the monoidal structures strict, whereas in the present work we are keeping all monoidal structures weak for consistency. In fact in \cite{gri1} it is proved that in a weak $n$-category a maximum of one dimension can have strict unit laws (besides the top dimension) with the resulting associated string diagrams still being a presheaf category, which is hypothesised to be an indicator of the $n$-category being equivalent to a fully weak one.




\section{Higher-order Eckmann--Hilton}\label{higherorder}


We will now provide the higher order Eckmann--Hilton argument. Our framework is a 3-fold internal monoid in \Cat; what we really use is the characterisation from Prop~\ref{threepsmon} which says we have a category equipped with 3 monoidal structures, any two of which form an object of $\Mon^2(\Cat)$.

We will represent the objects of the category as cubes, and the 3 monoidal structures as horizontal, vertical, and ``depthal'' stacking, represented as follows:

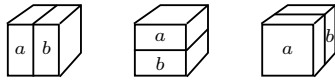
\[\cubeasix \hs{2} \cubeaeight \hs{2} \cubec\]

\noi We will write the unit constraints as follows:

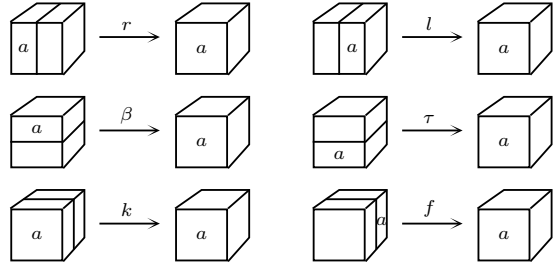
\[\renewcommand{\arraystretch}{2}
\setlength{\arraycolsep}{12pt}
\begin{array}{cc}

\cubealeft \tmapcube{r} \cubeatotal & \cubearight \tmapcube{l} \cubeatotal  \\
\cubeatop \tmapcube{\beta} \cubeatotal & \cubeabottom \tmapcube{\tau} \cubeatotal \\
\cubeafront \tmapcube{k} \cubeatotal & \cubeaback \tmapcube{f} \cubeatotal
\end{array}\]

\noi As we have strict interchange we can combine these tensor products by dividing the cubes up in any of the three directions simultaneously. The notation just needs some attention when a label is needed at the bottom-back-left corner, which we will annotate like this:

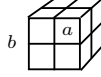
\[\cubebone\]

\noi We will write tensor products of maps similarly. Note that we will draw the arrows for constraints in the directions defined above, but as they are all isomorphisms we can read them in the inverse direction. Also, when no instances of depthal tensor products are involved we may still write the objects in 2-dimensional form as before.


%

\begin{theorem}\label{ehsphere} Let $C$ be a 3-monoidal category. Thus we already know it has a braiding given by:

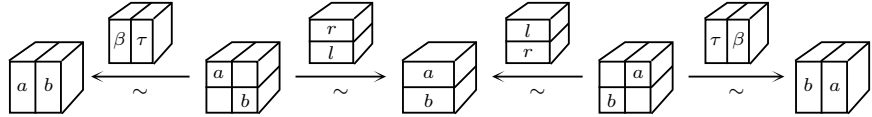
\[\psset{unit=1.3mm}
\pspicture(0,-5)(80,5)

\rput(0,0){\rnode{a1}{\cubeasix}}

\rput(20,0){\rnode{a2}{\cubeaseven}}

\rput(40,0){\rnode{a3}{\cubeaeight}}

\rput(60,0){\rnode{a4}{\cubeaone}}

\rput(80,0){\rnode{a5}{\cubeatwo}}

\psset{arrows=->, nodesep=6pt,labelsep=4pt}
\ncline[nodesepA=3pt, arrows=<-,]{a1}{a2} \naput[labelsep=4pt]{\cubemapfive}\nbput{\scr $\sim$}
\ncline{a2}{a3} \naput{\cubemapsix}\nbput{\scr $\sim$}
\ncline{<-}{a3}{a4} \naput{\cubemapsixb}\nbput{\scr $\sim$}
\ncline[nodesepB=4pt, arrows=->,]{a4}{a5} \naput[labelsep=4pt]{\cubemapfiveb}\nbput{\scr $\sim$}

\endpspicture\]

\noi This braiding is now forced to be a symmetry.
\end{theorem}

To show that the braiding in question is a symmetry, we need to show that the ``over-crossing'' and ``under-crossing'' are the same; this amounts to showing that the complete circle of \eha\ commutes, or that the two paths from 9 o'clock to 3 o'clock are equal: via 12 o'clock or via 6 o'clock.  Thus we need to fill in the circle, and in fact the circle can be filled in in two ways using a third dimension: a ``front argument'' and a ``back argument'', making a sphere. As these are strictly commutative diagrams it doesn't matter whether we go around the front or the back, but each half divides into four, so in all we will divide the sphere into octants.

In order to show that each octant commutes, we need the following key lemma.  There are four versions of the lemma, as it involves interaction between left and right unit constraints, and top and bottom unit constraints.  Although many of the morphisms in the following diagrams are identities, we find it clearer to write them all out in order to see more precisely what is going on. The proofs all follow similarly but as they are not complicated we find it helpful to make them explicit.

\begin{lemma}\label{keylemma}
In a 2-monoidal category (with our usual notation), the following diagrams commute for any objects $a$ and $b$.

\[
\psset{unit=0.13cm,labelsep=3pt,nodesep=4pt}
\pspicture(-30,-40)(30,40)
\begin{small}


\rput(-20,20){
\pspicture(-10,-24)(10,14)



\rput(-10,10){\rnode{a1}{$\vtp{\htp{H}{V}\vs{0.2}}{\htp{H}{a}\hs{0.2}}$}}  
\rput(10,10){\rnode{a2}{$\htp{\vtp{H}{H}}{\vtp{V}{a}}$}}  
\rput(-10,-10){\rnode{a3}{$\vtp{V}{a}$}}  
\rput(10,-10){\rnode{a4}{$\htp{H}{a}$}}  

\rput(0,0){\rnode{c}{$\htp{H}{\vtp{V}{a}}$}}  
\rput(0,-20){\rnode{b}{$a$}}  

\ncline{->}{a1}{a2} \naput{{\scriptsize $\chi$}} 
\ncline[nodesepA=1pt,nodesepB=2pt]{->}{a2}{c} \nbput[labelsep=0pt]{{\scriptsize $\htp{\phi}{1}$}} 
\ncline[nodesepA=0pt]{->}{c}{a3} \naput{{\scriptsize $l$}} 
\ncline[nodesepA=2pt,nodesepB=2pt]{->}{c}{a4} \nbput[labelsep=1pt]{{\scriptsize $\htp{1}{\tau}$}} 

\ncline{->}{a3}{b} \nbput{{\scriptsize $\tau$}} 
\ncline{->}{a4}{b} \naput{{\scriptsize $l$}} 

\ncline{->}{a1}{a3} \nbput{{\scriptsize $\vtp{l}{l}$}} 
\ncline{->}{a2}{a4} \naput{{\scriptsize $\htp{\phi}{\tau}$}} 
\eps}


\rput(20,20){
\pspicture(-10,-24)(10,14)



\rput(-10,10){\rnode{a1}{$\vtp{\htp{V}{H}\vs{0.2}}{\hs{0.2}\htp{a}{H}}$}}  
\rput(10,10){\rnode{a2}{$\htp{\vtp{V}{a}}{\vtp{H}{H}}$}}  
\rput(-10,-10){\rnode{a3}{$\vtp{V}{a}$}}  
\rput(10,-10){\rnode{a4}{$\htp{a}{H}$}}  

\rput(0,0){\rnode{c}{$\htp{\vtp{V}{a}}{H}$}}  
\rput(0,-20){\rnode{b}{$a$}}  

\ncline{->}{a1}{a2} \naput{{\scriptsize $\chi$}} 
\ncline[nodesepA=1pt,nodesepB=-1pt]{->}{a2}{c} \nbput[labelsep=0pt]{{\scriptsize $\htp{1}{\phi}$}} 
\ncline[nodesepA=0pt]{->}{c}{a3} \naput{{\scriptsize $r$}} 
\ncline[nodesepA=2pt,nodesepB=2pt]{->}{c}{a4} \nbput[labelsep=1pt]{{\scriptsize $\htp{\tau}{1}$}} 

\ncline{->}{a3}{b} \nbput{{\scriptsize $\tau$}} 
\ncline{->}{a4}{b} \naput{{\scriptsize $r$}} 

\ncline{->}{a1}{a3} \nbput{{\scriptsize $\vtp{r}{r}$}} 
\ncline{->}{a2}{a4} \naput{{\scriptsize $\htp{\tau}{\phi}$}} 
\eps}


\rput(-20,-20){
\pspicture(-10,-24)(10,14)



\rput(-10,10){\rnode{a1}{$\vtp{\htp{H}{a}\hs{0.2}\vs{0.2}}{\htp{H}{V}}$}}  
\rput(10,10){\rnode{a2}{$\htp{\vtp{H}{H}}{\vtp{a}{V}}$}}  
\rput(-10,-10){\rnode{a3}{$\vtp{a}{V}$}}  
\rput(10,-10){\rnode{a4}{$\htp{H}{a}$}}  

\rput(0,0){\rnode{c}{$\htp{H}{\vtp{a}{V}}$}}  
\rput(0,-20){\rnode{b}{$a$}}  

\ncline{->}{a1}{a2} \naput{{\scriptsize $\chi$}} 
\ncline[nodesepA=1pt,nodesepB=2pt]{->}{a2}{c} \nbput[labelsep=0pt]{{\scriptsize $\htp{\phi}{1}$}} 
\ncline[nodesepA=0pt]{->}{c}{a3} \naput{{\scriptsize $l$}} 
\ncline[nodesepA=2pt,nodesepB=2pt]{->}{c}{a4} \nbput[labelsep=1pt]{{\scriptsize $\htp{1}{\beta}$}} 

\ncline{->}{a3}{b} \nbput{{\scriptsize $\beta$}} 
\ncline{->}{a4}{b} \naput{{\scriptsize $l$}} 

\ncline{->}{a1}{a3} \nbput{{\scriptsize $\vtp{l}{l}$}} 
\ncline{->}{a2}{a4} \naput{{\scriptsize $\htp{\phi}{\beta}$}} 
\eps}


\rput(20,-20){
\pspicture(-10,-24)(10,14)



\rput(-10,10){\rnode{a1}{$\vtp{\htp{a}{H}\vs{0.2}}{\htp{V}{H}\hs{0.2}}$}}  
\rput(10,10){\rnode{a2}{$\htp{\vtp{a}{V}}{\vtp{H}{H}}$}}  
\rput(-10,-10){\rnode{a3}{$\vtp{a}{V}$}}  
\rput(10,-10){\rnode{a4}{$\htp{a}{H}$}}  

\rput(0,0){\rnode{c}{$\htp{\vtp{a}{V}}{H}$}}  
\rput(0,-20){\rnode{b}{$a$}}  

\ncline{->}{a1}{a2} \naput{{\scriptsize $\chi$}} 
\ncline[nodesepA=1pt,nodesepB=-1pt]{->}{a2}{c} \nbput[labelsep=0pt]{{\scriptsize $\htp{1}{\phi}$}} 
\ncline[nodesepA=0pt]{->}{c}{a3} \naput{{\scriptsize $r$}} 
\ncline[nodesepA=2pt,nodesepB=2pt]{->}{c}{a4} \nbput[labelsep=1pt]{{\scriptsize $\htp{\beta}{1}$}} 

\ncline{->}{a3}{b} \nbput{{\scriptsize $\beta$}} 
\ncline{->}{a4}{b} \naput{{\scriptsize $r$}} 

\ncline{->}{a1}{a3} \nbput{{\scriptsize $\vtp{r}{r}$}} 
\ncline{->}{a2}{a4} \naput{{\scriptsize $\htp{\beta}{\phi}$}} 
\eps}

\end{small}
\endpspicture
\]

\noi Note that in fact with our level of strictness:

\begin{itemize}

\item $\chi$ and $\phi$ are the identity
\item $V=H$ (which we will write as 1), and
\item $l$ and $r$ are strict on $V$
\end{itemize}

\noi so these diagrams become:

\[
\psset{unit=0.1cm,labelsep=3pt,nodesep=4pt}
\pspicture(-30,-35)(30,35)
\begin{small}


\rput(-20,17){
\pspicture(-10,-14)(10,14)



\rput(-10,10){\rnode{a1}{$\fourtp{1}{1}{1}{a}$}}  
\rput(10,10){\rnode{a2}{$\htp{1}{a}$}}  
\rput(-10,-10){\rnode{a3}{$\vtp{1}{a}$}}  
\rput(10,-10){\rnode{a4}{$a$}}  

\ncline{->}{a1}{a2} \naput{{\scriptsize $\htp{1}{\tau}$}} 
\ncline{->}{a3}{a4} \nbput{{\scriptsize $\tau$}} 
\ncline[nodesepB=1pt]{->}{a1}{a3} \nbput{{\scriptsize $\vtp{1}{l}$}} 
\ncline{->}{a2}{a4} \naput{{\scriptsize $l$}} 

\endpspicture
}


\rput(20,17){
\pspicture(-10,-14)(10,14)



\rput(-10,10){\rnode{a1}{$\fourtp{1}{1}{a}{1}$}}  
\rput(10,10){\rnode{a2}{$\htp{a}{1}$}}  
\rput(-10,-10){\rnode{a3}{$\vtp{1}{a}$}}  
\rput(10,-10){\rnode{a4}{$a$}}  

\ncline{->}{a1}{a2} \naput{{\scriptsize $\htp{\tau}{1}$}} 
\ncline{->}{a3}{a4} \nbput{{\scriptsize $\tau$}} 
\ncline[nodesepB=1pt]{->}{a1}{a3} \nbput{{\scriptsize $\vtp{1}{r}$}} 
\ncline{->}{a2}{a4} \naput{{\scriptsize $r$}} 

\endpspicture
}


\rput(-20,-17){
\pspicture(-10,-14)(10,14)



\rput(-10,10){\rnode{a1}{$\fourtp{1}{a}{1}{1}$}}  
\rput(10,10){\rnode{a2}{$\htp{1}{a}$}}  
\rput(-10,-10){\rnode{a3}{$\vtp{1}{a}$}}  
\rput(10,-10){\rnode{a4}{$a$}}  

\ncline{->}{a1}{a2} \naput{{\scriptsize $\htp{1}{\beta}$}} 
\ncline{->}{a3}{a4} \nbput{{\scriptsize $\beta$}} 
\ncline[nodesepB=1pt]{->}{a1}{a3} \nbput{{\scriptsize $\vtp{l}{1}$}} 
\ncline{->}{a2}{a4} \naput{{\scriptsize $l$}} 

\endpspicture
}


\rput(20,-17){
\pspicture(-10,-14)(10,14)



\rput(-10,10){\rnode{a1}{$\fourtp{a}{1}{1}{1}$}}  
\rput(10,10){\rnode{a2}{$\htp{a}{1}$}}  
\rput(-10,-10){\rnode{a3}{$\vtp{a}{1}$}}  
\rput(10,-10){\rnode{a4}{$a$}}  

\ncline{->}{a1}{a2} \naput{{\scriptsize $\htp{\beta}{1}$}} 
\ncline{->}{a3}{a4} \nbput{{\scriptsize $\beta$}} 
\ncline[nodesepB=1pt]{->}{a1}{a3} \nbput{{\scriptsize $\vtp{r}{1}$}} 
\ncline{->}{a2}{a4} \naput{{\scriptsize $r$}} 

\endpspicture
}

\end{small}
\eps\]
\end{lemma}

\begin{proof}
In each diagram (in its original, less streamlined form) the top left region is one of the four ``Key Axioms'' for 2-monoidal categories (as in Proposition~\ref{twofoldaxioms}), the upper right triangle is the functoriality of the monoidal structure in question, and the bottom region is naturality  of $l$ or $r$. \end{proof}

We are now ready to show that each octant commutes. Note that the following proof is straightforwardly adapted to any of the octants, it just involves a different combination of left/right, top/bottom, front/back; this is why there are 8 of them.

Each one is a commuting diagram of isomorphisms. We choose to write each isomorphism as an arrow pointing in the direction of the removal of a unit, as we think this makes it clearer.

\begin{lemma}[Eckmann--Hilton Octant]\label{ehoctant} In any 3-monoidal category (with our usual notation), for any objects $a$ and $b$ the following diagram commutes.

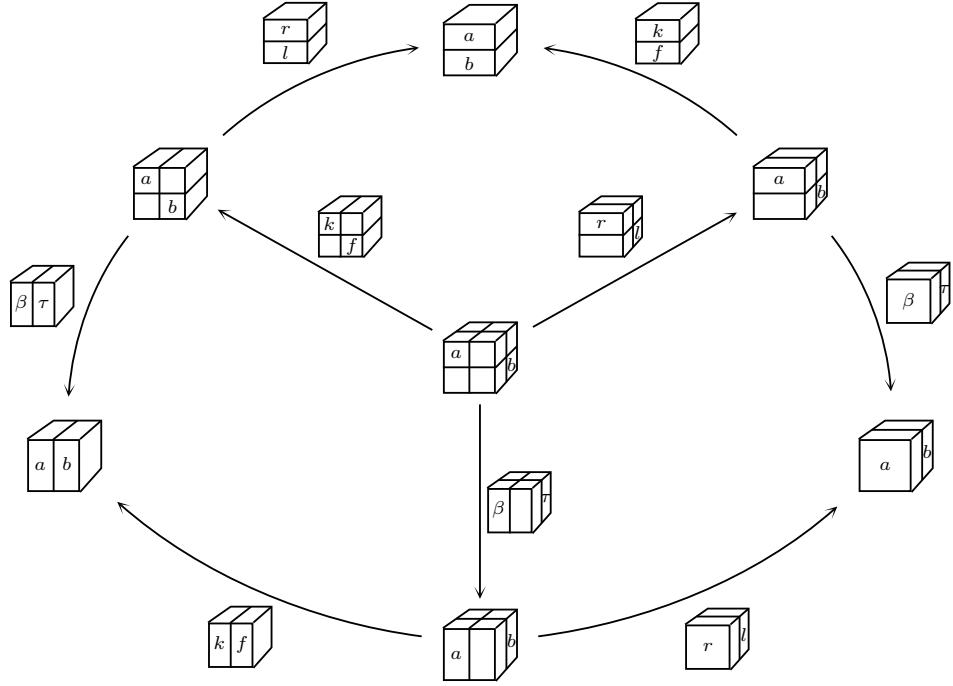
\[\pspicture(-60,-30)(60,60)
\polar


\rput(0,55){\rnode{a8}{\cubeaeight}}
\rput(-55,0){\rnode{a6}{\cubeasix}}
\rput(-41,36){\rnode{a7}{\cubeaseven}}

\rput(41,36){\rnode{b8}{\cubebeight}}
\rput(0,-25){\rnode{b6}{\cubebsix}}
\rput(0,13){\rnode{b7}{\cubebseven}}

\rput(55,0){\rnode{c}{\cubec}}

\psset{arrows=->, nodesep=6pt,arcangle=17}

\ncarc[nodesepA=6pt,nodesepB=9pt]{a7}{a8} \naput{\cubemapsix}
\ncarc[nodesepA=8pt,nodesepB=6pt]{b6}{a6} \naput{\cubemapfour}
\ncarc[nodesepA=6pt,nodesepB=10pt]{b8}{c} \naput{\cubemaptwo}

\psset{arcangle=-17}
\ncarc[nodesepA=6pt,nodesepB=8pt]{a7}{a6} \nbput{\cubemapfive}
\ncarc[nodesepA=6pt,nodesepB=10pt]{b8}{a8} \nbput{\cubemapone}
\ncarc[nodesepA=8pt,nodesepB=8pt]{b6}{c} \nbput{\cubemapthree}


\psset{arrows=-}

\psset{arrows=->, nodesep=6pt,arcangle=17}
\ncline[nodesepA=6pt,nodesepB=6pt]{b7}{b8} \naput{\cubemapseven}
\ncline[nodesepA=4pt,nodesepB=4pt]{b7}{a7} \nbput{\cubemapnine}

\psset{arcangle=-17}
\ncline[nodesepA=4pt,nodesepB=3pt]{b7}{b6} \naput{\cubemapeight}

\psset{nodesep=4pt}

%

\eps\]

\end{lemma}

\begin{proof}
We will examine the bottom right region; the others will follow similarly. By functoriality of the depthal tensor product we can examine the front and back halves of the cubes separately. They become the following squares:

\[
\psset{unit=0.1cm,labelsep=3pt,nodesep=4pt}
\pspicture(-30,-17)(30,17)
\begin{small}


\rput(20,0){
\pspicture(-10,-14)(10,14)



\rput(0,0){\sf back}

\rput(-10,10){\rnode{a1}{$\fourtp{1}{1}{1}{b}$}}  
\rput(-10,-10){\rnode{a2}{$\htp{1}{b}$}}  
\rput(10,10){\rnode{a3}{$\vtp{1}{b}$}}  
\rput(10,-10){\rnode{a4}{$a$}}  

\ncline[nodesepB=4pt]{->}{a1}{a2} \nbput{{\scriptsize $\htp{1}{\tau}$}} 
\ncline{->}{a3}{a4} \naput{{\scriptsize $\tau$}} 
\ncline[nodesepA=7pt,nodesepB=6pt]{->}{a1}{a3} \naput[labelsep=0pt]{{\scriptsize $\vtp{1}{l}$}} 
\ncline{->}{a2}{a4} \nbput{{\scriptsize $l$}} 

\endpspicture
}


\rput(-20,0){
\pspicture(-10,-14)(10,14)


\rput(0,0){\sf front}


\rput(-10,10){\rnode{a1}{$\fourtp{a}{1}{1}{1}$}}  
\rput(-10,-10){\rnode{a2}{$\htp{a}{1}$}}  
\rput(10,10){\rnode{a3}{$\vtp{a}{1}$}}  
\rput(10,-10){\rnode{a4}{$a$}}  

\ncline{->}{a1}{a2} \nbput{{\scriptsize $\htp{\beta}{1}$}} 
\ncline{->}{a3}{a4} \naput{{\scriptsize $\beta$}} 
\ncline[nodesepA=7pt,nodesepB=6pt]{->}{a1}{a3} \naput[labelsep=0pt]{{\scriptsize $\vtp{r}{1}$}} 
\ncline{->}{a2}{a4} \nbput{{\scriptsize $r$}} 

\endpspicture
}

\end{small}
\eps\]

\noi which are each instances of Lemma~\ref{keylemma}. \end{proof}

We are now ready to prove that the \eh\ sphere commutes.

\begin{proof}[Proof of Theorem~\ref{ehsphere}] \emph{(\eh\ Sphere)} We need to show that the outside of the following diagram commutes; we see it commutes by filling it with four instances of the Eckmann--Hilton Octant as shown.

%
%
%
%
%
%
%
%
%
%
%
%
%
%
%
%
%
%
%
%
%

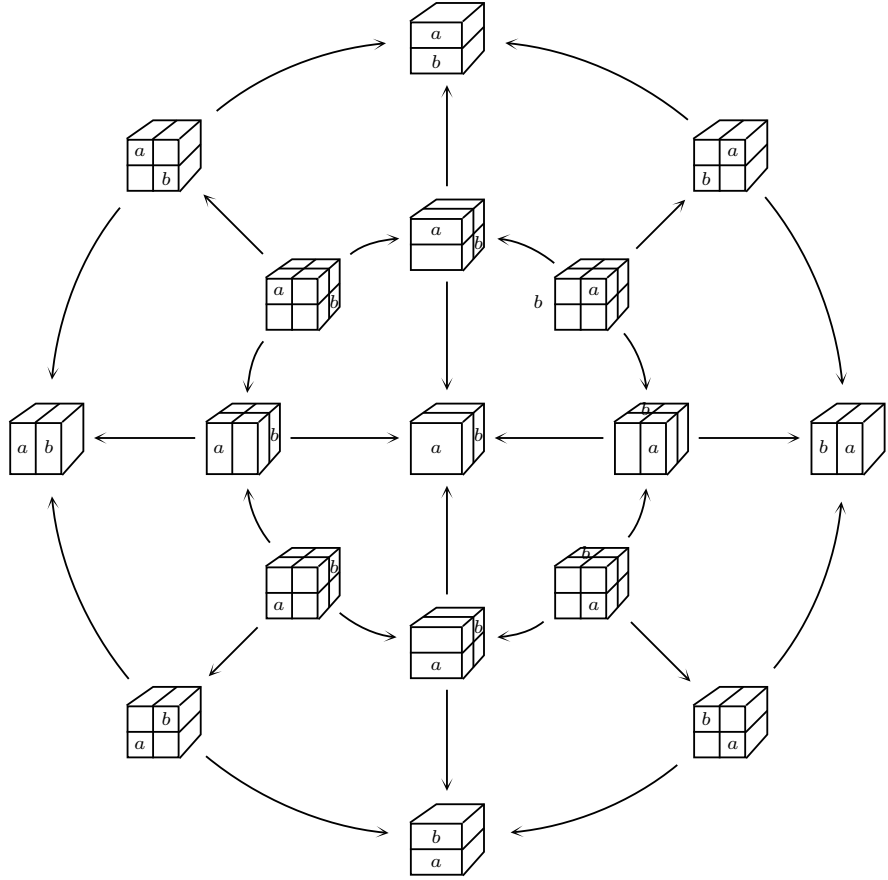
\[\pspicture(-60,-60)(60,60) \label{spherediag}
\polar

\rput(53;90){\rnode{a8}{\cubeaeight}}
\rput(53;45){\rnode{a1}{\cubeaone}}
\rput(53;0){\rnode{a2}{\cubeatwo}}
\rput(53;-45){\rnode{a3}{\cubeathree}}
\rput(53;-90){\rnode{a4}{\cubeafour}}
\rput(53;-135){\rnode{a5}{\cubeafive}}
\rput(53;180){\rnode{a6}{\cubeasix}}
\rput(53;135){\rnode{a7}{\cubeaseven}}

\psset{arrows=->, nodesep=6pt,arcangle=17}

\ncarc[nodesepA=2pt,nodesepB=6pt]{a1}{a2}
\ncarc[nodesepA=6pt,nodesepB=10pt]{a3}{a4}
\ncarc[nodesepA=2pt,nodesepB=8pt]{a5}{a6}
\ncarc[nodesepA=6pt,nodesepB=9pt]{a7}{a8}

\psset{arcangle=-17}
\ncarc[nodesepA=2pt,nodesepB=8pt]{a1}{a8}
\ncarc[nodesepA=6pt,nodesepB=8pt]{a7}{a6}
\ncarc[nodesepA=2pt,nodesepB=8pt]{a5}{a4}
\ncarc[nodesepA=6pt,nodesepB=10pt]{a3}{a2}


\psset{arrows=-}

\rput(27;90){\rnode{b8}{\cubebeight}}
\rput(27;45){\rnode{b1}{\cubebone}}
\rput(27;0){\rnode{b2}{\cubebtwo}}
\rput(27;-45){\rnode{b3}{\cubebthree}}
\rput(27;-90){\rnode{b4}{\cubebfour}}
\rput(27;-135){\rnode{b5}{\cubebfive}}
\rput(27;180){\rnode{b6}{\cubebsix}}
\rput(27;135){\rnode{b7}{\cubebseven}}

\rput(0,0){\rnode{c}{\cubec}}

\psset{arrows=->, nodesep=6pt,arcangle=17}

\ncarc[nodesepA=1pt,nodesepB=4pt]{b1}{b2}
\ncarc[nodesepA=4pt,nodesepB=5pt]{b3}{b4}
\ncarc[nodesepA=1pt,nodesepB=5pt]{b5}{b6}
\ncarc[nodesepA=4pt,nodesepB=4pt]{b7}{b8}

\psset{arcangle=-17}
\ncarc[nodesepA=0pt,nodesepB=5pt]{b1}{b8}
\ncarc[nodesepA=4pt,nodesepB=3pt]{b7}{b6}
\ncarc[nodesepA=0pt,nodesepB=5pt]{b5}{b4}
\ncarc[nodesepA=3pt,nodesepB=5pt]{b3}{b2}

\psset{nodesep=4pt}

\ncline[nodesep=3pt]{b1}{a1}
\ncline{b2}{a2}
\ncline[nodesep=1pt]{b3}{a3}
\ncline{b4}{a4}
\ncline[nodesep=3pt]{b5}{a5}
\ncline{b6}{a6}
\ncline[nodesep=1pt]{b7}{a7}
\ncline{b8}{a8}

\ncline{b2}{c}
\ncline{b4}{c}
\ncline{b6}{c}
\ncline{b8}{c}

\eps\]

\end{proof}

Note that there is a ``front proof'' and a ``back proof'', just like for basic \eh\ there is a ``top proof'' and a ``bottom proof''; in this case the back proof would proceed similarly but with the object $a$ at the back of $b$, and it would use the other four Eckmann--Hilton Octants. In higher dimensions, with appropriate weakness, these would be the two syllepses between the under-crossing and the over-crossing.

Note also that we don't need the full coherence of a 3-monoidal category in order for the Eckmann--Hilton Sphere to commute, just like associativity is not needed for the basic \eh\ argument; in both cases we never need three-fold products in the same direction, so we don't need the coherence related to associativity.

%
%
%
%
%
%


\section{$n$-degenerate $(n+1)$-categories}\label{degenncat}

In this section we will establish our definition of semi-strict $n$-category, and then show that for $n \geq 3$ any $n$-degenerate semi-strict $(n+1)$-category is, in particular, a 3-monoidal category, and therefore a symmetric monoidal category. Essentially we are beginning the definition of a comparison 2-functor
\[\cat{nDegen(n+1)Cat} \tra \cat{SymMonCat}\]
giving its action on objects.  In future work we will complete the definition of the 2-functor and show that it is a biequivalence. In previous work \cite{cc1, cc2} we studied the case $n=2$ which involved a comparison 2-functor
\[\cat{2Degen3Cat} \tra \cat{BrMonCat}\]
from 2-degenerate 3-categories to braided monoidal categories. The fact that we are now simultaneously dealing with all $n\geq 3$ is a small part of the stabilisation hypothesis. 

Our definition of semi-strict $n$-category is by iterated weak enrichment, as defined in Section~\ref{weakenrichment}. We then show that the $n$-degenerate ones are $n$-monoidal categories, via the definition by iterated internalisation, and hence are symmetric monoidal categories.


\subsection{Background on semi-strict $n$-categories}
\label{semistrict}

%
%
%
%
%


There are many possible versions of ``semi-strictness'' for $n$-categories. The idea for this work is to have weak composition but strict interchange; we achieve this by iterating the weak enrichment construction with respect to strict maps. In what follows we will often refer to semi-strict $n$-categories simply as $n$-categories, as this is the only type of $n$-category we will consider for the rest of this work; however, 2-category totalities are always strict. Recall that in Definition~\ref{kwcat} we defined, for any cartesian 2-category $K$, the 2-category \cat{$K$-wCat} of categories weakly enriched in $K$, strict functors, and icon-like transformations.

\begin{definition}
For each $n \geq 2$ we define a 2-category \cat{$n$-Cat} of semi-strict $n$-categories, strict maps, and transformations, inductively as follows.

\begin{itemize}

\item $\cat{$2$-Cat} = \cat{Cat-wCat}$ (which we have previously called $\cat{Bicat}_s$)

\item For $n \geq 2$, $\cat{$(n+1)$-Cat} = \cat{($n$-Cat)-wCat}$

\end{itemize}

\end{definition}

\noi Note that this means $\cat{$3$-Cat} = \cat{Cat-wCat-wCat}$ which has weak 1-composition and weak 0-composition, but strict interchange. This is different from the semi-strict 3-categories we studied in \cite{cc1, cc2}, which were defined as $\cat{Cat-wCat-Cat}$, and had weak 1-composition but strict 0-composition and interchange. We can think of the 2-category of semi-strict $n$-categories as
\[\cat{Cat-wCat-wCat- $\cdots$ -wCat}\]

\noi In a future work we will show that these can alternatively be constructed using iterated distributive laws.

Note that in our more streamlined notation we will write \cat{$n$-Cat} as $\Cat^n$. So

\begin{itemize}

\item $\Cat^2 = \cat{Cat(Cat)} = \cat{$2$-Cat} = \cat{Cat-wCat} = \cat{Bicat}_s$

\item For $n \geq 2$, $\Cat^{n+1} = \Cat(\Cat^n)$

\end{itemize}

\noi As we will never be taking products of $\Cat$ in this work, this will not be ambiguous, and greatly aids the notation.

\subsection{Degeneracy and enrichment}

We will now show that in the degenerate case, iterated weak enrichment in $K$ corresponds to taking iterated internal pseudomonoids in $K$. This is immediate (or taken as the definition) when we only do it one time; we need to check that the equivalence carries through in the iteration.

First we need to define iterated degeneracy carefully. Recall (Definition~\ref{dcatk}) that we have defined $d(\Cat(K))$ to be the full sub-2-category of $\Cat(K)$ whose 0-cells are degenerate, that is, the underlying $K$-graph has only one object. It may be tempting to blithely write $d^n$ but we need to be careful about both where this applies and how it is constructed: degeneracy can be iterated as many times as enrichment has taken place, and each time an enrichment has happened we can ask for it to be degenerate. Further to our previous streamlined notation we will use the following: for $n > 1$, we write $\Cat^n(K) := \Cat\big(\Cat^{n-1}(K)\big)$. We can then define iterated degeneracy. The notation is a little impenetrable but the idea is that at any stage of enrichment we have a set of 0-cells and a hom-object. The last level of degeneracy asks for those 0-cells to be trivial, and the remaining levels ask for the hom-object to be degenerate, providing we already have a definition of degeneracy there.



\begin{definition}
Let $K$ be a cartesian 2-category. Then we make the following definition by induction: 

\begin{itemize}

\item $d^1\big(\Cat^1(K)\big) = d\big(\Cat(K)\big)$

\item for any $k > 1$: $d^k\big(\Cat^k(K)\big) = d(\Cat(d^{k-1}(\Cat^{k-1}(K))))$
\end{itemize}

\end{definition}


%
%
%
%
%
%
%

\noi We can define $k$-degenerate $n$-categories as a special case. While we only need the cases $k = n-1$ in this work, we will give the more general definition here for completeness. The idea is that an $n$-category is $k$-degenerate if

\begin{itemize}
\item it is degenerate as an $\big(\cat{$(n-1)$-Cat}\big)$-category, and
\item its hom $(n-1)$-category is $(k-1)$-degenerate.
\end{itemize}

\begin{definition}
We define for every $1 \leq k < n$ a 2-category $d^k(\Cat^n)$ of $k$-degenerate $n$-categories, strict maps, and transformations, as follows:

\begin{itemize}
 \item for $1 = k < n$: $d^k(\Cat^n) = d(\Cat^n) := d\big(\Cat(\Cat^{n-1})\big)$

 \item for $1 < k < n$: $d^k(\Cat^n) := d\big(\Cat(d^{k-1}(\Cat^{n-1}))\big)$

\end{itemize}

\end{definition}

\noi We can now immediately deduce the relationship between iterated enrichment, iterated degeneracy, and iterated internal monoids. Recall (Definition~\ref{psmon}) we defined the internal pseudomonoids in $K$ to be degenerate $K$-categories where we have performed a dimension shift to ignore the single 0-cell. Thus
\[d\big(\Cat(K)\big) \catequiv \Mon(K)\]
via a particularly strict 2-equivalence that is strictly functorial, surjective on objects, and locally an isomorphism.

\begin{proposition}
Let $K$ be a cartesian 2-category. Then for all $k \geq 1$ there is a 2-equivalence of 2-categories
\[d^k\big(\Cat^k(K)\big) \catequiv \Mon^k(K)\]
that is surjective on objects and locally an isomorphism.
\end{proposition}

\begin{proof} We proceed by induction. The case $k=1$ follows from the definition. Now for $k>1$ suppose the result is true for $k-1$.  Note that since $\Mon$ is a 2-functor $\cat{2-Cat}_c \tra \cat{2-Cat}_c$ it preserves the equivalences in question.

We have
\[\renewcommand{\arraystretch}{1.2}
\begin{array}{rcll}
d^k\big(\Cat^k(K)\big) & = & d(\Cat(d^{k-1}(\Cat^{k-1}(K)))) & \mbox{by definition of $d^k$} \\
&\catequiv& \Mon(d^{k-1}(\Cat^{k-1}(K))) & \mbox{by definition of \Mon} \\
&\catequiv& \Mon(\Mon^{k-1}(K)) & \mbox{by induction hypothesis} \\
&=& \Mon^k(K) & \mbox{by definition}
\end{array}\]

\noi Each of the equivalences is surjective on objects and locally an isomorphism, so the composite is also.

\end{proof}


\begin{theorem}\label{mainthmcat}
Let $n \geq 1$. Then an $n$-degenerate $(n+1)$-category is an $n$-fold monoidal category, and thus by Theorem~\ref{ehsphere} a symmetric monoidal category.
\end{theorem}

\begin{proof}
By the previous result, with $K=\Cat$, we have

%
%

\[\renewcommand{\arraystretch}{1.3}
\begin{array}{rcl}
d^n\left(\Cat^{n+1}\right) &=& d^n\big(\Cat^n(\Cat)\big) \\
&\catequiv& \Mon^n(\Cat)
\end{array}\]

\noi Then by Proposition~\ref{ntothree} we have an object of $\Mon^3(\Cat)$, and the result follows.

\end{proof}

\section{Future work}\label{futurework}

\subsection{Totalities}

The eventual aim of this work is to make a full comparison of 2-category totalities; essentially in this work we have constructed the action on 0-cells of a comparison functor:
\[\cat{nDegen(n+1)Cat} \tra \cat{SymMonCat}\]

\subsection{Distributive laws}

One interesting approach to the underlying constructions of this work is to use 2-distributive laws, similarly to the work of \cite{cc2}. As this was not directly necessary for this work we have left it for a future work.

\subsection{Higher-dimensional generalisations}

There should be generalisations of the Eckmann--Hilton Sphere into higher dimensions; we could increase the number of monoidal structures or increase the number of ambient dimensions. The number of monoidal structures determines the number of dimensions of the resulting ``Eckmann--Hilton sphere'', which then determines what type of extra structure the argument produces for us on the ambient $n$-category.

\bc
\begin{tabular}{c|c|c}
no. of & dims of &  resulting structure  \\
monoidal structures & EH sphere & from EH \\[2pt]
\hline && \\[-10pt]
2     &      1 &            braiding (= 1-lepsis) \\
3     &      2 &            syllepsis (= 2-lepsis) \\
4     &      3 &            3-lepsis               \\
5     &      4 &            4-lepsis               
\end{tabular}
\ec

Having $k$ monoidal structures gives us $2k$ tensor products of any pair of objects $a$ and $b$ (order matters) producing $2k$ vertices on the sphere, and dividing the sphere into $2^k$-ants which are essentially $(k-1)$-simplices; the octants in this work are 2-simplices (triangles) if we consider the 3 directions of tensor product as the vertices.

The number of ambient dimensions determines how many monoidal structures we can fit in before the resulting structure is forced to be a symmetry. If we have $n$ ambient dimensions (so we're a $k$-monoidal $n$-category) then any simplex up to $n$ dimensions can be a non-trivial cell thus giving us a front and a back syllepsis for the front and back halves of the sphere respectively.

Any simplex higher than $n$ dimensions must be an equality, forcing an equality between a front and back syllepsis thereby resulting in a symmetry. Thus we get a symmetry if $k-1>n$, that is $k \geq n+2$.

The ``periodic table'' for $k$-monoidal $n$-categories is more straightforward than the one for $k$-degenerate $(n+k)$ categories, as the stabilisation happens vertically rather than diagonally; the relationship between the two tables is a rotation by 45\degree.

\bc
\begin{tabular}{l|lllllc}
& 1-cat & 2-cat & 3-cat & 4-cat & 5-cat & $\cdots$ \\[2pt]
\hline\\[-10pt]
1-monoidal & monoidal & monoidal & monoidal & monoidal & monoidal \\
2-monoidal & braided & braided & braided & braided & braided \\
3-monoidal & symmetric & sylleptic & 2-leptic & 2-leptic & 2-leptic \\
4-monoidal & \hs{1.5}\textquotedbl & symmetric & 3-leptic & 3-leptic & 3-leptic\\
5-monoidal & \hs{1.5}\textquotedbl & \hs{1.5}\textquotedbl & symmetric & 4-leptic & 4-leptic\\
6-monoidal &\hs{1.5}\textquotedbl & \hs{1.5}\textquotedbl & \hs{1.5}\textquotedbl & symmetric & 5-leptic\\
7-monoidal &\hs{1.5}\textquotedbl & \hs{1.5}\textquotedbl & \hs{1.5}\textquotedbl & \hs{1.5}\textquotedbl  & symmetric
\end{tabular}
\ec

One further possible generalisation of the Eckmann--Hilton Sphere is to include weak interchange. This introduces more ``faces'' to the sphere, and in particular introduces permutohedra in the middle, coming from all the different possible orders in which to perform the $n$ monoidal products (or $n$ directions of composition).



\newcommand{\etalchar}[1]{$^{#1}$}

\ed